\documentclass[11pt,twoside]{article}
\usepackage{mathrsfs}
\usepackage{amssymb}
\usepackage{amsmath}
\usepackage[all]{xy}
\usepackage{amsthm}

\usepackage{url}
\numberwithin{equation}{section}

\newtheorem{theorem}{Theorem}[section]
\newtheorem{proposition}[theorem]{Proposition}
\newtheorem{lemma}[theorem]{Lemma}

\newtheorem{corollary}[theorem]{Corollary}

\newtheorem{theorem*}{Theorem}

\theoremstyle{definition}
\newtheorem{definition}[theorem]{Definition}
\newtheorem{example}[theorem]{Example}

\theoremstyle{remark}
\newtheorem{remark}[theorem]{Remark}

\usepackage{paralist}

\usepackage{indentfirst}
\newcommand{\Hom}{\operatorname{Hom}}
\newcommand{\Ext}{\operatorname{Ext}}
\newcommand{\E}{\mathbb{E}}

\newcommand{\op}{\operatorname{op}}

\newcommand{\ra}{\rightarrow}

\newcommand{\s}{\mathfrak{s}}

\newcommand{\C}{\mathfrak{C}}

\newcommand{\add}{\operatorname{add}}

\newcommand{\End}{\operatorname{End}}
\newcommand{\Id}{\operatorname{Id}}
\renewcommand{\Im}{\operatorname{Im}}
\newcommand{\Ker}{\operatorname{Ker}}
\newcommand{\rad}{\operatorname{rad}}
\newcommand{\CC}{\mathcal{C}}

\newcommand{\CP}{\mathcal{P}}

\newcommand{\Cco}{\overline\C}
\newcommand{\Ccu}{\underline\C}
\newcommand{\set}[1]{\left\{#1\right\}}
\newcommand{\Endc}{\End_\C}
\newcommand{\radc}{\rad_\C}

\title{ \bf Almost Split Triangles and Morphisms Determined by Objects in Extriangulated Categories
}
\vspace{0.2cm}

\author{Tiwei Zhao$^a$, \ Lingling Tan$^a$, \
Zhaoyong Huang$^{b,}$\thanks{E-mail: tiweizhao@qfnu.edu.cn (Corresponding Author), tanll@qfnu.edu.cn, huangzy@nju.edu.cn}\\
{\it \footnotesize $^a$School of Mathematical Sciences, Qufu Normal University, Qufu 273165, Shandong Province, P.R. China;}\\
{\it \footnotesize $^b$Department of Mathematics, Nanjing University, Nanjing 210093, Jiangsu Province, P.R. China}}
\date{ }
\begin{document}

\baselineskip=15pt
\maketitle

\begin{abstract}
Let $(\C,\E,\s)$ be an $\Ext$-finite, Krull-Schmidt and $k$-linear extriangulated category with $k$ a
commutative artinian ring. We define an additive subcategory $\C_r$ (respectively, $\C_l$) of $\C$
in terms of the representable functors from the stable category of $\C$ modulo $\s$-injectives
(respectively, $\s$-projectives) to $k$-modules, which
consists of all $\s$-projective (respectively, $\s$-injective) objects and objects isomorphic to direct summands of finite
direct sums of all third (respectively, first) terms of
almost split $\s$-triangles. We investigate the subcategories $\C_r$ and $\C_l$ in terms of morphisms determined
 by objects, and then give equivalent characterizations on the existence of almost split $\s$-triangles.
\end{abstract}

\pagestyle{myheadings}
\markboth{\rightline {\scriptsize   T. Zhao, L. Tan, Z. Huang}}
         {\leftline{\scriptsize  Almost Split Triangles and Morphisms  Determined by Objects in Extriangulated Categories}}

\section{Introduction} %delete * to number this section

In algebra, geometry and topology, exact categories and triangulated categories are two fundamental structures.
 As expected, exact categories and triangulated categories are not
independent of each other. A well-known fact is that triangulated categories which at the same time are
abelian must be semisimple \cite{Mi}. Also, there are a series of ways to produce triangulated
categories from abelian ones, such as, taking the stable categories of Frobenius exact categories \cite{Hap}, or taking the
homotopy categories or derived categories of complexes over abelian categories \cite{Mi}.

On the other hand, because of the recent
development of the cluster theory, it becomes possible to produce abelian categories from triangulated ones, that is,
starting from a cluster category and taking a cluster tilting subcategory, one can get a suitable quotient category,
which turns out to be abelian \cite{KZ,Na}. In addition, exact categories and triangulated categories share properties
in many homological invariants, for example, in the aspect of the approximation theory \cite{AbNa,Liu,Na}.
The approximation theory is the main part of relative homological algebra and the representation theory of algebras,
and its starting point is to approximate arbitrary objects by a class of
suitable subcategories. In this process, the notion of cotorsion pairs \cite{Liu,LN17,Na13} provides a fruitful context,
in particular, it is closely related to many important homological structures, such as $t$-structures, co-$t$-structures,
cluster tilting subcategories, and so on. In general, to transfer the homological properties between exact categories and
triangulated categories, one needs  to specify to the case of stable categories of Frobenius exact categories, and then
lift (or descend) the associated definitions and statements, and finally adapt the proof so that it can be applied to arbitrary exact
(or triangulated) categories.

However, it is not easy to do it in general case, especially in the third step above.
To overcome the difficulty, Nakaoka and Palu \cite{NP} introduced the notion of externally triangulated categories
(extriangulated categories for short) by a careful looking what is necessary in the definition of cotorsion pairs in
exact and triangulated cases. Under this notion, exact categories with a suitable assumption and extension-closed subcategories of triangulated
categories (they may no longer be triangulated
categories in general) both are externally triangulated (\cite{NP}), and hence, in some levels, it becomes easy to give uniform statements
and proofs in the exact and triangulated settings \cite{LN17,NP,ZH,ZZ}.

The Auslander-Reiten theory, initiated in \cite{AR1,AR2}, plays a crucial role in the representation theory
of algebras and related topics, especially in the aspect of understanding
the structure of module categories of finite dimensional algebras \cite{Au95} and that of exact and triangulated categories \cite{Ch,J,JL,LZ}.
As a simultaneous generalization and enhancement of the Auslander-Reiten theory in exact categories and triangulated categories,
recently, Iyama, Nakaoka and Palu \cite{INP} investigated the Auslander-Reiten theory in extriangulated categories.
They gave two different sets of sufficient conditions in the extriangulated category so that the existence of
almost split extensions is equivalent to that of an Auslander-Reiten-Serre duality. In this paper,
as a continuation of their work, we will investigate the existence of almost split triangles in extriangulated categories.
The paper is organized as follows.

In Section 2, we give some terminologies and some preliminary results.

Let $(\C,\E,\s)$ be an $\Ext$-finite, Krull-Schmidt and $k$-linear extriangulated category with $k$ a
commutative artinian ring. In Section 3, we introduce an additive subcategory $\C_r$ (respectively, $\C_l$) of $\C$
in terms of the representable functors from the stable category of $\C$ modulo $\s$-injectives
(respectively, $\s$-projectives) to the category of $k$-modules.
For an indecomposable object $Y$ in $\C$, if $Y$ is non-$\s$-projective (respectively, non-$\s$-injective),
then $Y\in\C_r$ (respectively, $Y\in\C_l$) if and only if there exists an almost split $\s$-triangle
ending (respectively, starting) at $Y$ (Proposition \ref{prop:C_r and C_l description}).
Moreover, we get two quasi-inverse functors $\tau$ and $\tau^-$ in the stable categories of $\C_r$ and $\C_l$
(Theorem \ref{prop:quasi-inverse}), and the pair $(\tau^-,\tau)$ forms an adjoint pair (Proposition \ref{adjoint}).

In Section 4, we mainly characterize the subcategory $\C_r$ via morphisms determined by objects.
We prove that for any $C\in\C_r$ and $Y\in\C$, if $H$ is a right $\End_{\C}(C)$-submodule of $\C(C,Y)$
containing the class $\mathcal{P}(C,Y)$ of $\s$-projective morphisms, then there exists an $\s$-triangle
$$\xymatrix@C=0.5cm{K\ar[r]&X\ar[r]^\alpha&Y\ar@{-->}[r]^\eta&},$$
such that $\alpha$ is right $C$-determined, $K\in\add(\tau C)$ and $H=\Im\C(C,\alpha)$ (Theorem \ref{thm:exist}).
Under the so-called WIC condition (see Section 4.2 for the definition), we get that an $\s$-deflation is
right $C$-determined for some object $C$ if and only if its intrinsic weak kernel lies in $\C_l$ (Theorem \ref{thm:det}).
Moreover, we have the following

\begin{theorem} {\rm (Theorem \ref{thm:C})}
Under the \emph{WIC} condition, the following statements are equivalent for any
non-$\s$-projective and indecomposable object $C$ in $\C$.
\begin{enumerate}
\item[$(1)$]\label{item:thm:C:1}
$C\in\C_r$.
\item[$(2)$]\label{item:thm:C:2}
For each object $Y$ and each right $\End_{\C}(C)$-submodule $H$ of $\C(C,Y)$ satisfying $\CP(C,Y)\subseteq H$, there exists a
right $C$-determined $\s$-deflation $\alpha\colon X\to Y$ such that $H=\Im\C(C,\alpha)$.
\setcounter{enumi}{3}
\item[$(3)$]\label{item:thm:C:4}
$C$ is an intrinsic weak cokernel of some $\s$-inflation $\alpha\colon X\to Y$ which is  left $K$-determined for some object $K$.
\item[$(4)$]\label{item:thm:C:5}
There exists an almost split $\s$-triangle ending at $C$.
\item[$(5)$]\label{item:thm:C:6}
There exists a non-retraction $\s$-deflation which is right $C$-determined.
\item[$(6)$]\label{item:thm:C:7}
There exists an $\s$-deflation $\alpha\colon X\to Y$ and a morphism $f\colon C\to Y$ such that $f$ almost factors through $\alpha$.
\end{enumerate}
\end{theorem}

In Section 5, we give some examples to illustrate the subcategories $\C_r$ and $\C_l$.

\section{Preliminaries}

Throughout  $\C$ is an additive category and $\E:\C^{\operatorname{op}}\times\C\rightarrow \mathfrak{A}b$
is a biadditive functor, where $\mathfrak{A}b$ is the category of abelian groups.

\subsection{$\E$-extensions}

\begin{definition}
(\cite[Definitions 2.1 and 2.5]{NP}) For any $A,C\in\C$, there exists a corresponding abelian group $\E(C,A)$.
\begin{enumerate}
\item[(1)] An element $\delta\in\E(C,A)$ is called an \emph{$\E$-extension}. More formally, an $\E$-extension is a triple $(A,\delta,C)$.
\item[(2)]The zero element $0$ in $\E(C,A)$ is called the \emph{split $\E$-extension}.
\end{enumerate}
\end{definition}

Let $a\in \C(A,A')$ and $c\in\C(C',C)$. Then we have the following commutative diagram
$$\xymatrix@C=0.5cm{
  \E(C,A) \ar[rr]^{\E(C,a)}\ar[d]^{\E(c,A)}\ar[rrd]^{\E(c,a)} && \E(C,A')\ar[d]^{\E(c,A')} \\
 \E(C',A)\ar[rr]^{\E(C',a)} &&\E(C',A') }
$$  in $\mathfrak{A}b$.
For an $\E$-extension $(A,\delta,C)$, we briefly write $a_\star\delta:=\E(C,a)(\delta)$ and
$c^\star\delta:=\E(c,A)(\delta)$. Then
$$\E(c,a)(\delta)=c^\star a_\star \delta= a_\star c^\star \delta.$$

\begin{definition}
(\cite[Definition 2.3]{NP})
Given two $\E$-extensions $(A,\delta,C)$ and $(A',\delta',C')$. A \emph{morphism} from $\delta$ to $\delta'$
is a pair $(a,c)$ of morphisms, where $a\in \C(A,A')$ and $c\in \C(C,C')$, such that $a_\star\delta=c^\star\delta$.
In this case, we write $(a,c):\delta\rightarrow \delta'.$
\end{definition}

Now let $A,C\in\C$. Two sequences of morphisms
$$
\xymatrix@C=0.5cm{
  A \ar[r]^x & B \ar[r]^y & C }
  \mbox{ and }
  \xymatrix@C=0.5cm{
  A \ar[r]^{x'} & B' \ar[r]^{y'} & C}
$$
are said to be \emph{equivalent} if there exists an isomorphism $b\in\C(B,B')$ such that the following diagram
$$
\xymatrix{
   A \ar[r]^x\ar@{=}[d] & B \ar[r]^y\ar[d]^b_\cong & C\ar@{=}[d]\\
    A \ar[r]^{x'} & B' \ar[r]^{y'} & C }
$$
commutes.
We denote by $[\xymatrix@C=0.5cm{
A \ar[r]^x & B \ar[r]^y & C }]$ the equivalence class of $\xymatrix@C=0.5cm{
A \ar[r]^x & B \ar[r]^y & C}$. In particular, we write $0:=[\xymatrix@C=0.5cm{
A \ar[r]^{\!\!\!\!\!\!\!{1\choose 0}} & A\oplus C \ar[r]^{~~~(0 \ 1)} & C }]$.

Note that, for any pair $\delta\in \E(C,A)$ and $\delta'\in \E(C',A')$, since $\E$ is biadditive,
there exists a natural isomorphism
$$\E(C\oplus C',A\oplus A')\cong \E(C,A)\oplus\E(C,A')\oplus\E(C',A)\oplus\E(C',A').$$
We define the symbol $\delta \oplus\delta'$ to be the element in $\E(C\oplus C',A\oplus A')$ corresponding to the element
$(\delta,0,0,\delta')$ in $\E(C,A)\oplus\E(C,A')\oplus\E(C',A)\oplus\E(C',A')$ through the above isomorphism.

\begin{definition}
 (\cite[Definition 2.9]{NP})
Let $\s$ be a correspondence which associates an equivalence class $\s(\delta)=[\xymatrix@C=0.5cm{
A \ar[r]^x & B \ar[r]^y & C }]$ to each $\E$-extension $\delta\in\E(C,A)$. The $\s$ is called a
\emph{realization} of $\E$ provided that it satisfies the following condition.
\begin{enumerate}
\item[(R)]  Let $\delta\in\E(C,A)$ and $\delta'\in\E(C',A')$ be any pair of $\E$-extensions with
$$\s(\delta)=[\xymatrix@C=0.5cm{
A \ar[r]^x & B \ar[r]^y & C }]
\mbox{ and }
\s(\delta')=[ \xymatrix@C=0.5cm{
A' \ar[r]^{x'} & B' \ar[r]^{y'} & C' }].$$
Then for any morphism $(a,c):\delta\rightarrow\delta'$, there exists $b\in\C(B,B')$ such that the following diagram
$$\xymatrix{
A \ar[r]^x\ar[d]^a & B \ar[r]^y\ar[d]^b & C\ar[d]^c\\
A' \ar[r]^{x'} & B' \ar[r]^{y'} & C' }$$
commutes.
\end{enumerate}
\end{definition}

Let $\s$ be a realization of $\E$. If $\s(\delta)=[\xymatrix@C=0.5cm{
A \ar[r]^x & B \ar[r]^y & C }]$ for some $\E$-extension $\delta\in\E(C,A)$, then we say that the sequence
$\xymatrix@C=0.5cm{A \ar[r]^x & B \ar[r]^y & C }$ \emph{realizes} $\delta$; and in the condition (R),
we say that the triple $(a,b,c)$ \emph{realizes} the morphism $(a,c)$.

For any two equivalence classes $[\xymatrix@C=0.5cm{
  A \ar[r]^x & B \ar[r]^y & C }]$ and $[\xymatrix@C=0.5cm{
  A' \ar[r]^{x'} & B' \ar[r]^{y'} & C' }]$, we define
$$[\xymatrix@C=0.5cm{
  A \ar[r]^x & B \ar[r]^y & C }]\oplus[\xymatrix@C=0.5cm{
  A' \ar[r]^{x'} & B' \ar[r]^{y'} & C' }]:=[\xymatrix@C=0.5cm{
  A\oplus A' \ar[r]^{x\oplus x'} & B\oplus B' \ar[r]^{y\oplus y'} & C\oplus C' }].$$

\begin{definition}
(\cite[Definition 2.10]{NP})
A realization $\s$ of $\E$ is called \emph{additive} if it satisfies the following conditions.
\begin{enumerate}
\item[(1)] For any $A,C\in\C$, the split $\E$-extension $0\in\E(C,A)$ satisfies $\s(0)=0$.
\item[(2)] For any pair of $\E$-extensions $\delta\in\E(C,A)$ and $\delta'\in\E(C',A')$,
we have $\s(\delta\oplus\delta')=\s(\delta)\oplus\s(\delta')$.
\end{enumerate}
\end{definition}

Let $\s$ be an additive realization of $\E$. By \cite[Remark 2.11]{NP},  if the sequence
$\xymatrix@C=0.5cm{A \ar[r]^x & B \ar[r]^y & C }$ realizes $0$ in $\E(C,A)$, then $x$ is a section and $y$ is a retraction.

\subsection{Externally triangulated categories}

\begin{definition}
(\cite[Definition 2.12]{NP}) Let $\C$ be an additive category. The triple $(\C,\E,\s)$ is called an \emph{externally triangulated}
(or \emph{extriangulated} for short) category if it satisfies the following conditions.
 \begin{enumerate}
\item[(ET1)] $\E:\C^{\operatorname{op}}\times\C\rightarrow \mathfrak{A}b$
is a biadditive functor.
\item[(ET2)] $\s$ is an additive realization of $\E$.
\item[(ET3)] Let $\delta\in\E(C,A)$ and $\delta'\in\E(C',A')$ be any pair of $\E$-extensions with
$$
\s(\delta)=[\xymatrix@C=0.5cm{
  A \ar[r]^x & B \ar[r]^y & C }]
  \mbox{ and }
 \s(\delta')=[ \xymatrix@C=0.5cm{
  A' \ar[r]^{x'} & B' \ar[r]^{y'} & C' }].
$$
For any commutative diagram
$$
\xymatrix{
   A \ar[r]^x\ar[d]^a & B \ar[r]^y\ar[d]^b & C\\
    A' \ar[r]^{x'} & B' \ar[r]^{y'} & C' }
$$
in $\C$, there exists a morphism $(a,c):\delta\ra \delta'$ which is realized by the triple $(a,b,c)$.
\item[${\rm (ET3)^{op}}$]  Dual of (ET3).
\item[(ET4)]   Let $\delta\in\E(C,A)$ and $\rho\in\E(F,B)$ be any pair of $\E$-extensions with
$$
\s(\delta)=[\xymatrix@C=0.5cm{
  A \ar[r]^x & B \ar[r]^y & C }]
  \mbox{ and }
 \s(\rho)=[ \xymatrix@C=0.5cm{
  B \ar[r]^{u} & D \ar[r]^{v} & F }].
$$
Then there exist an object $E\in\C$, an $\E$-extension $\xi$ with $\s(\xi)=[\xymatrix@C=0.5cm{
  A \ar[r]^z & D \ar[r]^w & E }]$, and a commutative diagram
$$
\xymatrix{
   A \ar[r]^x\ar@{=}[d] & B \ar[r]^y\ar[d]^u & C\ar[d]^s\\
    A \ar[r]^{z} & D \ar[r]^{w}\ar[d]^v &E\ar[d]^t\\
    &F\ar@{=}[r] &F }
$$
in $\C$, which satisfy the following compatibilities.
 \begin{enumerate}
\item[(i)] $\s(y_\star\rho)=[\xymatrix@C=0.5cm{
  C \ar[r]^s & E \ar[r]^t & F }].$
\item[(ii)] $s^\star\xi=\delta$.
\item[(iii)] $x_\star\xi=t^\star\rho$.
\end{enumerate}
\item[${\rm (ET4)^{op}}$]   Dual of (ET4).
\end{enumerate}
\end{definition}

For examples of extriangulated categories, see \cite[Remark 3.3]{HZZ}, \cite[Example 2.13]{NP},
 \cite[Example 2.8]{ZH} and \cite[Corollary 4.12 and Remark 4.13]{ZZ}.

\begin{definition}
(\cite[Definition 1.16]{INP}) Let $(\C,\E,\s)$ be a triple satisfying (ET1) and (ET2).
\begin{enumerate}
\item[(1)] If a sequence $\xymatrix@C=0.5cm{
A \ar[r]^x & B \ar[r]^y & C }$ realizes an $\E$-extension $\delta\in\E(C,A)$, then the pair
\linebreak
$(\xymatrix@C=0.5cm{
A \ar[r]^x & B \ar[r]^y & C },\delta)$ is called an \emph{$\s$-triangle}, and write it in the following way
  $$
  \xymatrix@C=0.5cm{
  A \ar[r]^x & B \ar[r]^y & C\ar@{-->}[r]^\delta&. }
  $$
  In this case, $x$ is called an \emph{$\s$-inflation}, and $y$ is called an \emph{$\s$-deflation}.
\item[(2)] Let $
  \xymatrix@C=0.5cm{
  A \ar[r]^x & B \ar[r]^y & C\ar@{-->}[r]^\delta& }
  $ and $
  \xymatrix@C=0.5cm{
  A' \ar[r]^{x'} & B' \ar[r]^{y'} & C'\ar@{-->}[r]^{\delta'}& }
$ be any pair of $\s$-triangles. If a triple $(a,b,c)$ realizes $(a,c):\delta\ra\delta'$ as in the condition (R),
then we write it as
  $$
\xymatrix{
   A \ar[r]^x\ar[d]^a & B \ar[r]^y\ar[d]^b & C\ar[d]^c\ar@{-->}[r]^\delta& \\
    A' \ar[r]^{x'} & B' \ar[r]^{y'} & C' \ar@{-->}[r]^{\delta'}&, }
$$
and call the triple $(a,b,c)$ a \emph{morphism of  $\s$-triangles}.
\end{enumerate}
\end{definition}

\begin{remark}
Let $(\C,\E,\s)$ be a triple satisfying (ET1) and (ET2), and let
$\xymatrix@C=0.5cm{
  A \ar[r]^x & B \ar[r]^y & C\ar@{-->}[r]^\delta& }
  $ be an $\s$-triangle.
\begin{enumerate}
\item[(1)] For any $a\in\C(A,A')$,  there exists a morphism of  $\s$-triangles
$$
\xymatrix{
   A \ar[r]^x\ar[d]^a & B \ar[r]^y\ar[d] & C\ar@{=}[d]\ar@{-->}[r]^\delta& \\
    A' \ar[r]^{x'} & B' \ar[r]^{y'} & C \ar@{-->}[r]^{a_\star\delta}&. }
$$
\item[(2)]  For any $c\in\C(C',C)$,  there exists a morphism of  $\s$-triangles
$$
\xymatrix{
   A \ar[r]^{x'}\ar@{=}[d] & B' \ar[r]^{y'}\ar[d] & C'\ar[d]^c\ar@{-->}[r]^{c^\star\delta}& \\
    A \ar[r]^{x} & B \ar[r]^{y} & C \ar@{-->}[r]^{\delta}&.}
$$
\end{enumerate}
\end{remark}

The following lemma is used frequently in this paper.

\begin{lemma}\label{factor} {\rm (\cite[Corollary 3.5]{NP})}
Assume that $(\C,\E,\s)$ satisfies  (ET1),(ET2),(ET3),(ET3)$^{\mathrm{op}}$. Let
$$
\xymatrix{
   A \ar[r]^x\ar[d]^a & B \ar[r]^y\ar[d]^b & C\ar[d]^c\ar@{-->}[r]^\delta& \\
    A' \ar[r]^{x'} & B' \ar[r]^{y'} & C' \ar@{-->}[r]^{\delta'}&, }
$$
be any morphism of $\s$-triangles. Then the following statements are equivalent.
\begin{itemize}
\item[$(1)$] $a$ factors through $x$.
\item[$(2)$] $a_{\star}\delta=c^{\star}\delta^{\prime}=0$.
\item[$(3)$] $c$ factors through $y^{\prime}$.
\end{itemize}
In particular, in the case $\delta=\delta^{\prime}$ and $(a,b,c)=(\Id_A,\Id_B,\Id_C)$, we have
\[ x\ \text{is a section}\ \Leftrightarrow\ \delta\ \text{is split}\ \Leftrightarrow\ y\ \text{is a retraction}. \]
\end{lemma}

\begin{definition}\label{DefCoStable}
Let $(\C,\E,\s)$ be a triple satisfying (ET1) and (ET2).
\begin{itemize}
  \item [(1)] Let $f\in\C(C',C)$ be a morphism. We call $f$ an \emph{$\s$-projective morphism} if $\mathbb{E}(f,A)=0$,
and  an \emph{$\s$-injective morphism}  if $\mathbb{E}(A,f)=0$ for any $A\in\C$.
  \item [(2)] Let $C\in\C$. We call $C$  an \emph{$\s$-projective object} if the identity morphism $\Id_C$ is $\s$-projective, and
   an \emph{$\s$-injective object} if the identity morphism $\Id_C$ is $\s$-injective.
\end{itemize}
\end{definition}

We call an $\s$-triangle $\xymatrix@C=0.5cm{
  A \ar[r]^x & B \ar[r]^y & C\ar@{-->}[r]^\delta& }
  $ \emph{split} if $\delta$ is a split $\E$-extension.

\begin{lemma}\label{arb}
Let $(\C,\E,\s)$ be a triple satisfying (ET1), (ET2), (ET3) and $\mbox{(ET3)}^{\op}$, and
let $f\in\C(C',C)$ be a morphism. Then the following statements are equivalent.
\begin{itemize}
  \item[$(1)$] $f$ is $\s$-projective.
  \item[$(2)$] $f$ factors through any $\s$-deflation $y\colon B\to C$.
  \item[$(3)$] For any $\s$-triangle $\xymatrix@C=0.5cm{A\ar[r]^x&B\ar[r]^y&C\ar@{-->}[r]^\delta&}$,
  if there exists a morphism of $\s$-triangles
  \begin{equation}\label{A}
  \begin{split}
    \xymatrix{
   A \ar[r]^{x'}\ar@{=}[d] & B' \ar[r]^{y'}\ar[d]^g & C'\ar[d]^f\ar@{-->}[r]^{f^\star\delta}& \\
    A \ar[r]^{x} & B \ar[r]^{y} & C \ar@{-->}[r]^{\delta}&,}
  \end{split}
\end{equation}
  then the top $\s$-triangle is split.
\end{itemize}
\end{lemma}

\begin{proof}
(1) $\Leftrightarrow$ (3) It follows from the definition of $\s$-projective morphisms.

(3) $\Rightarrow$ (2) Since  $y'$ is a retraction by Lemma \ref{factor}, it follows that there exists a morphism $y'': C'\to B'$ such that
$y'\circ y''=\Id_{C'}$. Then $y\circ(g\circ y'')=f\circ y'\circ y''=f$, and hence (2) holds.

(2) $\Rightarrow$ (3) For any $\s$-triangle $\xymatrix@C=0.5cm{A\ar[r]^x&B\ar[r]^y&C\ar@{-->}[r]^\delta&}$,
consider the diagram (\ref{A}). By assumption, $f$
factors through $y$, and so $x'$ is a section by  Lemma \ref{factor}. Thus $f^\star\delta=0$, that is, the top $\s$-triangle is split.
\end{proof}

We denote by $\mathcal{P}$ (respectively, $\mathcal{I}$) the ideal of $\C$ consisting of all $\s$-projective
(respectively, $\s$-injective) morphisms.
The \emph{stable category} (respectively, \emph{costable category}) of $\C$ is defined as the ideal quotient
\[\underline{\C}:=\C/\mathcal{P}\ (\mbox{respectively,}\ \overline{\C}:=\C/\mathcal{I}).\]

\begin{lemma}\label{lem:ft.stable}
Let $\alpha\colon X \to Y$ be an $\s$-deflation. Then a morphism $f\colon T\to Y$ factors through $\alpha$ in $\C$
if and only if $\underline f$ factors through $\underline\alpha$ in $\underline\C$.
\end{lemma}

\begin{proof}
It suffices to show the sufficiency. Assume that $\underline f$ factors through $\underline\alpha$ in $\underline\C$.
Then there exists some morphism $g\colon T\to X$ in $\C$ such that $\underline f = \underline\alpha\circ\underline g$
in $\underline\C$. We have that $f-\alpha\circ g$ is $\s$-projective in $\C$. Since $\alpha$ is an $\s$-deflation,
there exists some morphism $h\colon T\to X$ such that $f-\alpha\circ g=\alpha\circ h$ in $\C$. It follows that
$f=\alpha\circ(g+h)$, factoring through $\alpha$ in $\C$.
\end{proof}

From now on, unless otherwise stated, we assume that the  extriangulated category $(\C,\E,\s)$ is an Ext-finite,
Krull-Schmidt and $k$-linear extriangulated category with $k$ a
commutative artinian ring. Here, an extriangulated category $(\C,\E,\s)$ is \emph{$k$-linear}
if $\C(A,B)$ and $\E(A,B)$ are $k$-modules such that the following compositions
\begin{eqnarray*}
&\C(A,B)\times\C(B,C)\to\C(A,C),&\\
&\C(A,B)\times\E(B,C)\times\C(C,D)\to\E(A,D),&
\end{eqnarray*}
are $k$-linear for any $A,B,C,D\in\C$; and is \emph{Ext-finite} if $\E(A,B)$ is a finitely generated $k$-module for any $A,B\in\C$.

\subsection{Almost split extensions}

In \cite{INP}, Iyama, Nakaoka and Palu introduced the notion of almost split $\E$-extensions.

\begin{definition}\label{DefARExt} (\cite[Definition 2.1]{INP})
A non-split (i.e. non-zero) $\E$-extension $\delta\in\E(C,A)$ is said to be {\it almost split} if it satisfies the following conditions.
\begin{itemize}
\item[{\rm (AS1)}] $a_{\star}\delta=0$ for any non-section $a\in\C(A,A^{\prime})$.
\item[{\rm (AS2)}] $c^{\star}\delta=0$ for any non-retraction $c\in\C(C^{\prime},C)$.
\end{itemize}
\end{definition}

\begin{definition}\label{DefARSEQ} (\cite[Definition 2.7]{INP})
Let $\xymatrix@C=0.5cm{A\ar[r]^x&B\ar[r]^y&C\ar@{-->}[r]^\delta&}$ be an $\s$-triangle in $\C$. It is called {\it almost split} if $\delta$ is an
almost split $\E$-extension.
\end{definition}

The following class of morphisms is basic to understand almost split $\s$-triangles.

\begin{definition} (\cite[Definition 2.8]{INP})
Let $\C$ be an additive category and $A$ an object in $\C$. A morphism $a\colon A\to B$ which is not a section is called \emph{left almost split} if
\begin{itemize}
\item[$\bullet$] any morphism $A\to B^{\prime}$ which is not a section factors through $a$.
\end{itemize}
Dually, a morphism $a\colon B\to A$ which is not a retraction is called \emph{right almost split} if
\begin{itemize}
\item[$\bullet$] any morphism $B^{\prime}\to A$ which is not a retraction factors through $a$.
\end{itemize}
\end{definition}

It was showed that for an $\s$-triangle $\xymatrix@C=0.5cm{A\ar[r]^x&B\ar[r]^y&C\ar@{-->}[r]^\delta&}$,
it is almost split if and only if $x$ is left almost split and $y$ is right almost split (\cite{INP}).

Recall that a morphism $f\colon X\to Y$ is called \emph{right minimal}, if each $g\in\Endc(X)$ with $f\circ g=f$ is an automorphism.
Dually, a morphism $f\colon X\to Y$ is called \emph{left minimal}, if each $g\in\Endc(Y)$ with $g\circ f=f$ is an automorphism.

By \cite[Propositions 2.5 and 2.10]{INP}, for an $\s$-triangle $\xymatrix@C=0.5cm{A\ar[r]^x&B\ar[r]^y&C\ar@{-->}[r]^\delta&}$, we have
\begin{itemize}
  \item if $x$ is left almost split, then $A$ is indecomposable and $y$ is right minimal;
  \item if $y$ is right almost split, then $C$ is indecomposable and $x$ is left minimal.
\end{itemize}
So if the  $\s$-triangle $\xymatrix@C=0.5cm{A\ar[r]^x&B\ar[r]^y&C\ar@{-->}[r]^\delta&}$ is almost split, then $x$  is left minimal,
$y$ is right minimal and $A,C$ are indecomposable.

\section{Two subcategories and two functors}

\subsection{Two subcategories $\mathfrak{C}_r$ and $\mathfrak{C}_l$}

Let $k$ be a commutative artinian ring and $\check{k}$ the minimal injective cogenerator for
the category $k$-mod of finitely generated $k$-modules. We write $D:=\Hom_k(-,\check{k})$.

The following two lemmas are essentially contained in \cite[Proposition 3.1]{INP} and its proof.

\begin{lemma}\label{lemma:representable}
Let $\C$ be a $k$-linear extriangulated category and
$$\xymatrix@C=0.5cm{X\ar[r]& E\ar[r]& Y\ar@{-->}[r]^{\eta}&}$$
an almost split $\s$-triangle, and let $\gamma\in D\E(Y,X)$ such that $\gamma(\eta)\neq 0$.
\begin{enumerate}
    \item[$(1)$] For each $M$, we have a non-degenerated $k$-bilinear map
      \[\langle-,-\rangle_M\colon\Cco(M,X)\times\E(Y,M)\longrightarrow\check{k}, \quad(\overline f,\mu)\mapsto\gamma(f_\star\mu).\]
      If moreover $\Cco(M,X)\in k\textnormal{-mod}$ for each $M$, then the induced map
      \[\phi_{Y,M}\colon\Cco(M,X)\longrightarrow D\E(Y,M), \quad\overline{f}\mapsto\langle\overline{f},-\rangle_M,\]
      is an isomorphism and natural in $M$ with $\gamma=\phi_{Y,X}(\overline{\Id_X})$.
    \item[$(2)$] For each $M$, we have a non-degenerated $k$-bilinear map
      \[\langle-,-\rangle_M\colon\E(M,X)\times\Ccu(Y,M)\longrightarrow\check{k},\quad(\mu,\underline g)\mapsto\gamma(g^\star\mu).\]
      If moreover $\Ccu(Y,M)\in k\textnormal{-mod}$ for each $M$, then the induced map
      \[\psi_{X,M}\colon\Ccu(Y,M)\longrightarrow D\E(M,X), \quad\underline{g}\mapsto\langle-,\underline{g}\rangle_M,\]
      is an isomorphism and natural in $M$ with $\gamma=\psi_{X,Y}(\underline{\Id_Y})$.
  \end{enumerate}
\end{lemma}

\begin{lemma}\label{lemma:a.s.s.}
Let $\C$ be a $k$-linear extriangulated category and $Y\in\C$ an indecomposable object.
\begin{enumerate}
\item[$(1)$] Assume $D\E(Y,Z)\in k\textnormal{-mod}$ for each $Z\in\C$. If the functor $D\E(Y,-)$
is isomorphic to $\Cco(-,Y')$ for some $Y'$, which has a non-$\s$-injective indecomposable direct summand,
then there exists an almost split $\s$-triangle ending at $Y$.
\item[$(2)$] Assume $D\E(Z,Y)\in k\textnormal{-mod}$ for each $Z\in\C$. If the functor $D\E(-,Y)$
is isomorphic to $\Ccu(Y',-)$ for some $Y'$, which has a non-$\s$-projective indecomposable direct summand,
then there exists an almost split $\s$-triangle starting at $Y$.
\end{enumerate}
\end{lemma}

We introduce two full subcategories of $\C$ as follows.
\[\C_r=\set{X\in\C\mid\text{the functor $D\E(X,-)\colon\overline\C\to k$-mod is representable}},\]
\[\C_l=\set{X\in\C\mid\text{the functor $D\E(-,X)\colon\underline\C\to k$-mod is representable}}.\]
Here, both $\C_r$ and $\C_l$ are additive subcategories which are closed under direct summands.
As a consequence of Lemmas~\ref{lemma:representable} and \ref{lemma:a.s.s.},
we have the following description of indecomposable objects in $\C_r$ and $\C_l$.

\begin{proposition}\label{prop:C_r and C_l description}
Let $Y$ be an indecomposable object in $\C$.
\begin{enumerate}
\item[$(1)$] If $Y$ is non-$\s$-projective, then $Y\in\C_r$ if and only if there exists an almost split $\s$-triangle ending at $Y$.
\item[$(2)$] If $Y$ is non-$\s$-injective, then $Y\in\C_l$ if and only if there exists an almost split $\s$-triangle starting at $Y$.
\end{enumerate}
\end{proposition}

We also have the following easy observation.

\begin{proposition}\label{lemma:closed under iso}
Let $X$ and $Y$ be two objects in $\C$.
\begin{enumerate}
\item[$(1)$] If $X\in\C_r$ and $X\simeq Y$ in $\underline\C$, then $Y\in\C_r$.
\item[$(2)$] If $X\in\C_l$ and $X\simeq Y$ in $\overline\C$, then $Y\in\C_l$.
\end{enumerate}
\end{proposition}

\begin{proof}
(1) The assertion follows from the fact that $D\E(X,-)\simeq D\E(Y,-)$ as functors.

(2) It is dual to (1).
\end{proof}

\subsection{Two functors $\tau$ and $\tau^{-}$}

For each $Y$ in $\C_r$, we define $\tau Y$ to be an object in $\C$ such that there exists an isomorphism of functors
\[\phi_Y\colon\Cco(-,\tau Y)\longrightarrow D\E(Y,-).\]
Then $\tau$ gives a map from the objects of $\C_r$ to that of $\C$. Dually, for each $X$ in $\C_l$,
we define $\tau^- X$ to be an object in $\C$ such that there exists an isomorphism of functors
\[\psi_X\colon\Ccu(\tau^-X,-)\longrightarrow D\E(-,X).\]
Then $\tau^-$ gives a map from the objects of $\C_l$ to that of $\C$.

\begin{lemma}\label{lemma:tau send iso to iso}
Let $X$ and $Y$ be two objects in $\C$.
\begin{enumerate}
\item[$(1)$] If $X,Y\in\C_r$ and $X\simeq Y$ in $\underline\C$, then $\tau X\simeq\tau Y$ in $\overline\C$.
\item[$(2)$] If $X,Y\in\C_l$ and $X\simeq Y$ in $\overline\C$, then $\tau^-X\simeq\tau^-Y$ in $\underline\C$.
\end{enumerate}
\end{lemma}

\begin{proof}
(1) We observe that $\Cco(-,\tau X)\simeq\Cco(-,\tau Y)$, since they are both isomorphic to $D\E(X,-)\simeq D\E(Y,-)$.
Then the assertion follows from the Yoneda's lemma.

(2) It is dual to (1).
\end{proof}

We denote by $\underline{\C_r}$ the image of $\C_r$ under the canonical functor $\C\to\underline\C$,
and by $\overline{\C_l}$ the image of $\C_l$ under the canonical functor $\C\to\overline\C$. Then we have

\begin{proposition}\label{lemma:Y_simeq_tau^- tau Y}
\begin{enumerate}
\item[]
\item[$(1)$] $\tau$ induces a functor from $\underline{\C_r}$ to $\overline{\C_l}$.
\item[$(2)$] $\tau^-$ induces a functor from $\overline{\C_l}$ to $\underline{\C_r}$.
\end{enumerate}
Moreover, we have
\begin{enumerate}
\item[$(3)$] If $Y\in\C_r$, then $Y\simeq\tau^-\tau Y$ in $\underline{\C_r}$.
\item[$(4)$] If $Y\in\C_l$, then  $Y\simeq\tau\tau^-Y$ in $\overline{\C_l}$.
\end{enumerate}
\end{proposition}

\begin{proof}
(1) Let $Y\in \C_r$. We may assume that $Y$ is indecomposable and non-$\s$-projective. By Lemma~\ref{prop:C_r and C_l description}(1),
there exists an almost split $\s$-triangle
$$\xymatrix@C=15pt{X\ar[r]& E\ar[r]& Y\ar@{-->}[r]^\eta&}.$$
Then we have $\Cco(-,X)\simeq D\E(Y,-)$ and $\Ccu(Y,-)\simeq D\E(-,X)$ by Lemma~\ref{lemma:representable}.
We then obtain $X\in\C_l$. It follows from the Yoneda's lemma that $\tau Y\simeq X$ in $\overline\C$, and $\tau^-X\simeq Y$
in $\underline\C$. So $\tau Y\in\C_l$ by Proposition~\ref{lemma:closed under iso}(2), and hence
$\tau^-\tau Y\simeq\tau^-X\simeq Y$ in $\underline\C$.
Here, the first isomorphism follows from Lemma~\ref{lemma:tau send iso to iso}(2).

For each morphism $f\colon Y\to Y'$ in $\C_r$, we define the morphism $\tau(f)\colon\tau Y\to \tau Y'$ in $\overline{\C_l}$
such that the following diagram commutes
\[\xymatrix{
  \Cco(-,\tau Y)\ar[r]^-{\phi_Y}\ar[d]_-{\Cco(-,\tau(f))}
    &D\E(Y ,-)\ar[d]^-{D\E(f,-)}\\
  \Cco(-,\tau Y')\ar[r]^-{\phi_Y'}
    &D\E(Y',-).
}\]
Here, the existence and uniqueness of $\tau(f)$ are guaranteed by the Yoneda's lemma.
Then it follows that $\tau$ is a functor from $\C_r$ to $\overline{\C_l}$. Moreover, if $f$ is $\s$-projective,
then $D\E(f,-)=0$ and thus $\tau(f)=0$ in $\overline{\C_l}$. Thus $\tau$ induces a functor from
$\underline{\C_r}$ to $\overline{\C_l}$ which we still denote by $\tau$.

(2) Similarly, we have a functor $\tau^-\colon\overline{\C_l}\to\underline{\C_r}$. For each $\overline g\colon X\to X'$ in $\overline{\C_l}$,
the morphism $\tau^-(\overline g)\colon \tau^-X\to\tau^-X'$ is given by the following commutative diagram
\[\xymatrix{
  \Ccu(\tau^-X',-)\ar[r]^-{\psi_{X'}}\ar[d]_-{\Ccu(\tau^-(\overline g),-)}
    &D\E(-,X')\ar[d]^-{D\E(-,g)}\\
  \Ccu(\tau^-X,-)\ar[r]^-{\psi_X}
    &D\E(-,X).
}\]

(3) Since $\tau Y\simeq X$ in $\overline\C$, we have  $\tau^-\tau Y\simeq\tau^-X\simeq Y$ in $\underline\C$ by
Lemma~\ref{lemma:tau send iso to iso}(2). Of course, $\tau^-\tau Y\simeq Y$ in $\underline{\C_r}$.

(4) It is similar to (3).
\end{proof}

For each $Y\in\underline{\C_r}$, set
\[\underline{\theta_Y}:=\psi_{\tau Y,Y}^{-1}(\phi_{Y,\tau Y}(\overline{\Id_{\tau Y}}))\colon \tau^-\tau Y\longrightarrow Y\]
in $\underline{\C_r}$. Dually, for each $X\in\overline{\C_l}$, set
\[\overline{\xi_X}:=\phi_{\tau^-X,X}^{-1} (\psi_{X,\tau^-X}(\underline{\Id_{\tau^-X}}))\colon X\longrightarrow \tau\tau^-X\]
in $\overline{\C_l}$.
The following result shows that the functors $\tau$ and $\tau^-$ are mutually quasi-inverse equivalences between
$\underline{\C_r}$ and $\overline{\C_l}$.

\begin{theorem}\label{prop:quasi-inverse}
Under the above definitions,
$$\underline\theta: \tau^-\tau\rightarrow \mbox{Id}_{\underline{\C_r}}\ and\
\overline\xi:\mbox{Id}_{\underline{\C_l}}\rightarrow\tau\tau^-$$ are both natural isomorphisms.
This implies that the functors $\tau$ and $\tau^-$ are quasi-inverse to each other.
\end{theorem}

\begin{proof} We only prove that $\underline\theta: \tau^-\tau\rightarrow \mbox{Id}_{\underline{\C_r}}$
is a natural isomorphism, the other is similar.

We first prove that $\underline\theta: \tau^-\tau\rightarrow \mbox{Id}_{\underline{\C_r}}$ is a natural transformation.
Indeed, for each $\underline f\colon Y\to Y'$ in $\underline{\C_r}$, consider the following diagram
  \[\xymatrix@+1.5em{
    \Cco(\tau Y,\tau Y)\ar[r]^-{\phi_{Y,\tau Y}}\ar[d]_-{\Cco(\tau Y,\tau(\underline f))}
      &D\E(Y,\tau Y)\ar[d]_-{D\E(f,\tau Y)}
      &\Ccu(\tau^-\tau Y,Y)\ar[l]_-{\psi_{\tau Y,Y}}\ar[d]_-{\Ccu(\tau^-\tau Y,\underline f)}\\
    \Cco(\tau Y,\tau Y')\ar[r]^-{\phi_{Y',\tau Y}}
      &D\E(Y',\tau Y)
      &\Ccu(\tau^-\tau Y,Y')\ar[l]_-{\psi_{\tau Y,Y'}}.
  }\]
The left square commutes by the definition of $\tau(\underline f)$, and the right square commutes since the isomorphism
$\psi_{\tau Y}$ is natural. By a diagram chasing, we have
\[\tau(\underline f)=\phi_{Y',\tau Y}^{-1}(\psi_{\tau Y,Y'}(\underline f\circ\underline{\theta_Y})).\]

We have the following commutative diagram
\[\xymatrix@+1.5em{
    \Cco(\tau Y',\tau Y')\ar[r]^-{\phi_{Y',\tau Y'}}\ar[d]_-{\Cco(\tau(\underline f),\tau Y')}
      &D\E(Y',\tau Y')\ar[d]_-{D\E(Y',\tau(\underline f))}
      &\Ccu(\tau^-\tau Y',Y')\ar[l]_-{\psi_{\tau Y',Y'}}\ar[d]_-{\Ccu(\tau^-\tau(\underline f),Y')}\\
    \Cco(\tau Y,\tau Y')\ar[r]^-{\phi_{Y',\tau Y}}
      &D\E(Y',\tau Y)
      &\Ccu(\tau^-\tau Y,Y')\ar[l]_-{\psi_{\tau Y,Y'}}.
  }\]
The right square commutes by the definition of $\tau^-\tau(\underline f)$. By a diagram chasing, we have
\[\tau(\underline f)=\phi_{Y',\tau Y}^{-1}(\psi_{\tau Y,Y'}(\underline{\theta_{Y'}}\circ\tau^-\tau(\underline f))).\]
We then obtain $\underline f\circ\underline{\theta_Y}=\underline{\theta_{Y'}}\circ\tau^-\tau(\underline f)$.
It follows that $\underline\theta$ is a natural transformation.

Next we prove that $\underline{\theta_Y}$ is an isomorphism for each $Y\in\C_r$. We may assume that $Y$ is
indecomposable and non-$\s$-projective in $\C$. Let $\alpha=\psi_{\tau Y,\tau^-\tau Y}(\underline{\Id_{\tau^-\tau Y}})$
in $D\E(\tau^-\tau Y,\tau Y)$ and let $\beta=\phi_{Y,\tau Y}(\overline{\Id_{\tau Y}})$ in $D\E(Y,\tau Y)$.
By the definition of $\underline{\theta_Y}$, we have $\beta = \psi_{\tau Y,Y} (\underline{\theta_Y})$.
Consider the following commutative diagram
  \[\xymatrix@C+1em{
    \Ccu(\tau^-\tau Y,\tau^-\tau Y)\ar[r]^-{\psi_{\tau Y,\tau^-\tau Y}}\ar[d]_-{\Ccu(\tau^-\tau Y,\underline{\theta_Y})}
      &D\E(\tau^-\tau Y,\tau Y)\ar[d]^-{D\E(\theta_Y,\tau Y)}\\
    \Ccu(\tau^-\tau Y,Y)\ar[r]^-{\psi_{\tau Y,Y}}
      &D\E(Y,\tau Y).
  }\]
Since $\psi_{\tau Y,\tau^-\tau Y}(\underline{\Id_{\tau^-\tau Y}})=\alpha$ and
$\Ccu(\tau^-\tau Y,\underline{\theta_Y})(\underline{\Id_{\tau^-\tau Y}})=\underline{\theta_Y}$, we have
\[\beta = D\E(\theta_Y,\tau Y)(\alpha) = \alpha\circ\E(\theta_Y,\tau Y).\]
By Lemma~\ref{prop:C_r and C_l description}(1), there exists an almost split $\s$-triangle
$$\xymatrix@C=0.5cm{X\ar[r]& E\ar[r]& Y\ar@{-->}[r]^{\eta}&}.$$
By Lemma~\ref{lemma:representable}(1), we have a natural isomorphism $\phi'\colon\Cco(-,X)\to D\E(Y,-)$
such that $\phi'_X(\overline{\Id_X})(\eta)$ $\ne 0$. Setting $\beta':=\phi'_X(\overline{\Id_X})$,
we have $\beta'(\eta)\neq 0$. By the Yoneda's lemma, there exists some $s\colon X\to\tau Y$ such that
$\Cco(-,s) = \phi_Y^{-1}\circ\phi'$. We obtain
  \[\beta'=\phi'_X(\overline{\Id_X})=(\phi_{Y,X}\circ\Cco(X,s))(\overline{\Id_X})=\phi_{Y,X}(\overline s).\]
Consider the following commutative diagram
  \[\xymatrix{
    \Cco(\tau Y,\tau Y)\ar[r]^-{\phi_{Y,\tau Y}}\ar[d]_-{\Cco(s,\tau Y)}
      &D\E(Y ,\tau Y)\ar[d]^-{D\E(Y,s)}\\
    \Cco(X,\tau Y)\ar[r]^-{\phi_{Y,X}}
      &D\E(Y,X).
  }\]
Since $\phi_{Y,\tau Y}(\overline{\Id_{\tau Y}})=\beta$ and $\Cco(s,\tau Y)(\overline{\Id_{\tau Y}})=\overline s$, we have
\[\beta' = D\E(Y,s)(\beta) = \beta\circ\E(Y,s)=\alpha\circ\E(\theta_Y,\tau Y)\circ\E(Y,s).\]
Thus we have
\[0\not = \beta'(\eta) = \alpha({\theta_Y}^\star(s_\star\eta)) = \alpha(s_\star({\theta_Y}^\star\eta)),\]
which implies that the $\s$-triangle
$$\xymatrix@C=0.5cm{X\ar[r]& E'\ar[r]& \tau^-\tau Y\ar@{-->}[r]^{\ \ \ {\theta_Y}^\star\eta}&}$$
is non-split. We claim that $\theta_Y\colon\tau^-\tau Y\to Y$ is a retraction in $\C$. Otherwise,
suppose that $\theta_Y\colon\tau^-\tau Y\to Y$ is not a retraction in $\C$. Since $\eta$ is almost split,
we have the following commutative diagram
   \[\xymatrix{
    X\ar[r]\ar@{=}[d] &E'\ar[r]\ar[d] &\tau^-\tau Y\ar@{-->}[r]^{{\theta_Y}^\star\eta}\ar[d]\ar@{.>}[ld]&\\
    X\ar[r]          &E \ar[r]       &Y\ar@{-->}[r]^{\eta}&.
  }\]
By Lemma \ref{factor}, the top $\s$-triangle is split, which is a contradiction.
Thus $\underline{\theta_Y}$ is an isomorphism in $\underline\C$ since we have already known $\tau^-\tau Y\simeq Y$
in $\underline\C$.
\end{proof}

The following proposition shows that the pair of functors $(\tau^-,\tau)$ forms an adjoint pair
with unit $\xi$ and counit $\theta$.

\begin{proposition}\label{adjoint}
We have
\begin{enumerate}
\item[$(1)$] $\tau(\underline{\theta_Y})\circ\overline{\xi_{\tau Y}}=\overline{\Id_{\tau Y}}$ for each $Y\in\underline{\C_r}$.
\item[$(2)$] $\underline{\theta_{\tau^-X}}\circ\tau^-(\overline{\xi_X})=\underline{\Id_{\tau^-X}}$ for each $X\in\overline{\C_l}$.
\end{enumerate}
\end{proposition}

\begin{proof}
We only prove (1). Consider the following  diagram
\[\xymatrix@+1em{
    \Cco(\tau Y,\tau\tau^-\tau Y)\ar[r]^-{\phi_{\tau^-\tau Y,\tau Y}}\ar[d]_-{\Cco(\tau Y,\tau(\underline{\theta_Y}))}
      &D\E(\tau^-\tau Y,\tau Y)\ar[d]_-{D\E(\theta_Y,\tau Y)}
      &\Ccu(\tau^-\tau Y,\tau^-\tau Y)\ar[l]_-{\psi_{\tau Y,\tau^-\tau Y}}\ar[d]_-{\Ccu(\tau^-\tau Y, \underline{\theta_Y})}\\
    \Cco(\tau Y,\tau Y)\ar[r]^-{\phi_{Y,\tau Y}}
      &D\E(Y,\tau Y)
      &\Ccu(\tau^-\tau Y,Y)\ar[l]_-{\psi_{\tau Y,Y}}.
  }\]
The left square commutes by the definition of $\tau(\underline{\theta_Y})$, and the right square commutes since
$\psi_{\tau Y}$ is natural. By the definitions of $\underline{\theta_Y}$ and $\overline{\xi_{\tau Y}}$, we have
\[\phi_{Y,\tau Y}^{-1}(\psi_{\tau Y,Y}(\underline{\theta_Y}))=\overline{\Id_{\tau Y}},\]
\[\phi_{\tau^-\tau Y,\tau Y}^{-1}(\psi_{\tau Y,\tau^-\tau Y}(\underline{\Id_{\tau^-\tau Y}}))=\overline{\xi_{\tau Y}}.\]
By the above commutative diagram, we have
\[\overline{\Id_{\tau Y}}=\Cco(\tau Y,\tau(\underline{\theta_Y}))(\overline{\xi_{\tau Y}})=
\tau(\underline{\theta_Y})\circ\overline{\xi_{\tau Y}}.\qedhere\]
\end{proof}

\section{Characterizing the subcategory $\C_r$ via morphisms determined by objects}

\subsection{$\s$-deflations determined by objects}

We first recall the concept of morphisms being determined by objects, which was introduced by Auslander in \cite{Au78} and closely related to
the Auslander-Reiten theory \cite{Au95} and Auslander bijections \cite{Ri13}. %Here, we use the terminologies in \cite[Section~3]{Ri13}.

\begin{definition} (\cite{Au78,Ri13})
Let $f\in \C(X, Y)$ and $C\in\C$. We call $f$ \emph{right $C$-determined} (or \emph{right determined} by $C$)
and call $C$ a \emph{right determiner} of $f$, if the following condition is satisfied: each
$g\in \C(T, Y)$ factors through $f$, provided that for each $h\in\C(C, T)$ the morphism $g\circ h$ factors through $f$.

If moreover $C$ is a direct summand of any right determiner of $f$, we call $C$ a \emph{minimal right determiner} of $f$.
\end{definition}

\begin{lemma}\label{lem:det.PB}
Consider a morphism of $\s$-triangles as follows:
\[
\xymatrix{
X\ar[r]^{e'}\ar@{=}[d]&Y'\ar[r]^{f'}\ar[d]^{g'}&Z'\ar@{-->}[r]^{\delta'}\ar[d]^g&\\
X\ar[r]^e&Y\ar[r]^{f}&Z\ar@{-->}[r]^{\delta}&.
}
\]
If $f$ is right $C$-determined for some object $C$, then $f'$ is also right $C$-determined.
\end{lemma}

\begin{proof}
Let $h\colon T\to Z'$ be a morphism such that for each $u\colon C\to T$, there exists some morphism $v\colon C\to Y'$
satisfying $h\circ u=f'\circ v$. Then we have $g\circ h\circ u=g\circ f'\circ v=f\circ g'\circ v$.
Since $f$ is right $C$-determined, there exists some morphism $s\colon T\to Y$ such that $g\circ h=f\circ s$.
We illustrate these by a commutative diagram as follows.
$$
\xymatrix@R=0.3cm{&&C\ar[d]^{^{ \forall u}}\ar@{..>}[ddl]_v&\\
&&T\ar@{..>}[ddl]^{\! \!s}\ar[d]^{^{\forall h}}&\\
X\ar[r]^{e'}\ar@{=}[d]&Y'\ar[r]^{f'}\ar[d]_{g'}&Z'\ar@{-->}[r]^{\delta'}\ar[d]^g&\\
X\ar[r]^e&Y\ar[r]^{f}&Z\ar@{-->}[r]^{\delta}&.}
$$
Thus, to show that $f'$ is right $C$-determined, it suffices to find  a morphism $t\colon T\to Y'$ such that $h=f'\circ t$.
Indeed, by \cite[Proposition 3.3]{NP}, we have the following commutative diagram with exact rows
$$
\xymatrix@C=1.3cm{\C(T,Y')\ar[r]^{\C(T,f')}\ar[d]_{\C(T,g')}&\C(T,Z')\ar[r]^{\delta'_\sharp}\ar[d]_{\C(T,g)}&\E(T,X)\ar@{=}[d]\\
\C(T,Y)\ar[r]^{\C(T,f)}&\C(T,Z)\ar[r]^{\delta_\sharp}&\E(T,X).}
$$
Since $\C(T,g)(h)=g\circ h=f\circ s=\C(T,f)(s)\in\Im\C(T,f)=\Ker\delta_\sharp$, we have
$$\delta'_\sharp(h)=\delta_\sharp (\C(T,g)(h))=0,$$
and hence $h\in {\rm Ker}\delta'_\sharp={\Im}\C(T,f')$. Thus there exists a morphism $t\colon T\to Y'$ such that $h=f'\circ t$.
\end{proof}

Given two objects $X$ and $Y$, we denote by $\radc(X,Y)$ the set of morphisms $f\colon X\to Y$,
that is, for any object $Z$ and any morphisms $g\colon Z\to X$ and $h\colon Y\to Z$,
the morphism $h\circ f\circ g$ lies in $\rad\Endc(Z)$. Then $\radc$ forms an ideal of $\CC$.
By \cite[Corollary~2.10]{Kr15}, we have
\begin{equation}\label{eq:rad}
\radc(X,Y) = \set{ f \colon X \to Y \middle|
f\circ g\in\rad\Endc(Y),
\mbox{ for each }
g\colon Y\to X}.
\end{equation}

A morphism $g\colon Z\to Y$ is said to \emph{almost factor through} $f\colon X\to Y$, if $g$ does not factor through $f$,
and for each object $T$ and each morphism $h\in\radc(T,Z)$, the morphism $g \circ h$ factors through $f$ (\cite{Ri12}).

\begin{proposition}\label{prop:ras}
Consider a morphism of $\s$-triangles as follows:
\[
\xymatrix{
X\ar[r]^{e'}\ar@{=}[d]&Y'\ar[r]^{f'}\ar[d]^{g'}&Z'\ar@{-->}[r]^{\delta'}\ar[d]^g&\\
X\ar[r]^e&Y\ar[r]^{f}&Z\ar@{-->}[r]^{\delta}&.
}
\]
If  $Z'$ is indecomposable and $g$ almost factors through $f$, then $f'$ is right almost split.
\end{proposition}

\begin{proof}
Since $g$ does not factors through $f$, $f'$ is not a retraction. Given an object $T$,
assume that $h\colon T\to Z'$ is not a retraction.
Since $Z'$ is indecomposable, $h\in\radc(T,Z')$ by (\ref{eq:rad}). Thus there exists  $s\in\C(T, Y)$
such that $g\circ h = f\circ s$. By using an argument similar to that in the proof of Lemma \ref{lem:det.PB},
there exists  $t\in\C(T,Y')$ such that $h=f'\circ t$. It follows that $f'$ is right almost split.
\end{proof}

\begin{lemma}\label{lem:det.stable}
Let $C$ be an object and $\alpha\colon X \to Y$ an $\s$-deflation.
Then the following statements are equivalent.
\begin{enumerate}
\item[$(1)$] $\alpha$ is right $C$-determined in $\C$.
\item[$(2)$] $\underline\alpha$ is right $C$-determined in $\underline\C$.
\end{enumerate}
\end{lemma}

\begin{proof}
$(1)\Rightarrow (2)$ Let $f\in\C(T,Y)$ such that for each $g\in\C(C,T)$,
the morphism $\underline{f} \circ \underline{g}$ factors through $\underline\alpha$ in $\underline\C$.
By Lemma~\ref{lem:ft.stable}, the morphism $f\circ g$ factors through $\alpha$ in $\C$. Since $\alpha$
is right $C$-determined in $\C$ by (1),  $f$ factors through $\alpha$ in $\C$.
It follows that $\underline f$ factors through $\underline\alpha$ in $\underline\C$, and hence
$\underline\alpha$ is right $C$-determined in $\underline\C$.

$(2)\Rightarrow (1)$ Let $f\in\C(T,Y)$  such that for each $g\in\C(C,T)$,
the morphism $f\circ g$ factors through $\alpha$ in $\C$. Then $\underline f \circ \underline g$
factors through $\underline\alpha$ in $\underline\C$. Since $\underline\alpha$ is right $C$-determined
in $\underline\C$ by (2), we have that $\underline f$ factors through $\underline\alpha$ in $\underline\C$.
By Lemma~\ref{lem:ft.stable}, the morphism $f$ factors through $\alpha$ in $\C$.
It follows that $\alpha$ is right $C$-determined in $\C$.
\end{proof}

\begin{proposition}\label{prop:det}
Let $C,C'\in\C$ such that $C\simeq C'$ in $\underline\C$. Then an $\s$-deflation $\alpha\colon X \to Y$
is right $C$-determined if and only if it is right $C'$-determined.
\end{proposition}

\begin{proof}
Note that $\underline\alpha$ is right $C$-determined in $\underline\C$ if and only if $\underline\alpha$ is right
$C'$-determined in $\underline\C$. Then the assertion follows by applying Lemma~\ref{lem:det.stable} twice.
\end{proof}

The following lemma generalizes \cite[Corollary~3.5]{Ri13}.

\begin{lemma}\label{lem:ker.det}
Let $$\xymatrix@C=0.5cm{K\ar[r]&X\ar[r]^\alpha&Y\ar@{-->}[r]^\eta&}$$
be an $\s$-triangle with $K\in\C_l$. Then $\alpha$ is right $\tau^-K$-determined.
\end{lemma}

\begin{proof}
Let $f\in\C(T,Y)$ such that for each $g\in\C(\tau^-K,T)$, the morphism $f \circ g$ factors through $\alpha$, that is,
there exists the following commutative diagram
$$
\xymatrix@C=0.5cm@R=0.3cm{&&\tau^-K\ar[d]^{^{ \forall g}}\ar@{..>}[ddl]^{ \circlearrowleft}&\\
&&T\ar[d]^{^{\forall f}}&\\
K\ar[r]&X\ar[r]^\alpha&Y\ar@{-->}[r]^\eta&.}
$$
By Lemma \ref{factor}, $(f \circ g)^\star\eta=0$. Since $K\in\C_l$, there exists a natural isomorphism
\[
\phi\colon \Ccu(\tau^-K,-)
\longrightarrow D\E(-,K).
\]
Set $\gamma: = \phi_{\tau^-K}(\underline{\Id_{\tau^-K}})$.
By the naturality of $\phi$, we have the following commutative diagram
\[\xymatrix@C+1em{
\Ccu(\tau^-K,\tau^-K)\ar[r]^-{\phi_{\tau^-K}}\ar[d]_-{\Ccu(\tau^-K,\underline{g})}
&D\E(\tau^-K,K)\ar[d]^-{D\E(g,K)}\\
\Ccu(\tau^-K,T)\ar[r]^-{\phi_{T}}
&D\E(T,K).
}\]
Thus
\[
\phi_T(\underline g)
= D\E(g,K)(\gamma)
= \gamma\circ\E(g,K),
\]
and hence
\[
\phi_T(\underline g)(f^\star\eta)
= \gamma(g^\star f^\star\eta)
= \gamma((f\circ g)^\star\eta)
= 0.
\]
Note that $\phi_T(\underline g)$ runs over all maps in $D\E(T,K)$, when $\underline g$ runs over all morphisms in $\Ccu(\tau^-K,T)$.
It follows that $f^\star\eta=0$ and hence, by Lemma \ref{factor}, the morphism $f$ factors through $\alpha$, that is,
there exists the following commutative diagram
$$
\xymatrix@C=0.5cm@R=0.3cm{&&\tau^-K\ar[d]^{^{ \forall g}}\ar[ddl]^{\!\!^{\Downarrow}}&\\
&&T\ar@{..>}[dl]^{\  \circlearrowleft}\ar[d]^{^{\forall f}}&\\
K\ar[r]&X\ar[r]&Y\ar@{-->}[r]^\eta&.}
$$
Thus $\alpha$ is right $\tau^-K$-determined.
\end{proof}

Note that $\CP(C,Y)$ is the subset of $\C(C,Y)$ consisting of all $\s$-projective morphisms.
The following existence theorem generalizes \cite[Corollary~XI.3.4]{Au95} to extriangulated categories.

\begin{theorem}\label{thm:exist}
Let $C\in\C_r$ and $Y\in\C$, and let $H$ be a right $\End_{\C}(C)$-submodule of $\C(C,Y)$ satisfying $\CP(C,Y)\subseteq H$.
Then there exists an $\s$-triangle
$$\xymatrix@C=0.5cm{K\ar[r]&X\ar[r]^\alpha&Y\ar@{-->}[r]^\eta&},$$
such that $\alpha$ is right $C$-determined, $K\in\add(\tau C)$ and $H=\Im\C(C,\alpha)$.
\end{theorem}

\begin{proof}
By assumption, $\tau C\in\C_l$ and $\tau^-\tau C\simeq C$ in $\underline\C$. Then there is a natural isomorphism
\[
\phi\colon\Ccu(C,-)
\longrightarrow D\E(-,\tau C).
\]
Set $\gamma:=\phi_C(\underline{\Id_C})$. By the naturality of $\phi$, for each object $Z$ and each $f\in\C(C,Z)$,
we have the following commutative diagram
\[\xymatrix@C+1em{
\Ccu(C,C)\ar[r]^-{\phi_C}\ar[d]_-{\Ccu(C,\underline{f})}
&D\E(C,\tau C)\ar[d]^-{D\E(f,\tau C)}\\
\Ccu(C,Z)\ar[r]^-{\phi_Z}
&D\E(Z,\tau C).
}\]
So
\[
\phi_Z(\underline f)
= D\E(f,\tau C)(\gamma)
= \gamma\circ\E(f,\tau C).
\]
Then for each $\mu\in\E(Z,\tau C)$, we have
\[
\phi_Z(\underline f)(\mu)
= \gamma(\E(f,\tau C)(\mu))
= \gamma(f^\star\mu).
\]

Set $\underline H:=H/\CP(C,Y)$ and
\[
\underline H^\perp:
= \set{\mu\in\E(Y,\tau C) \middle|
\phi_Y(\underline h)(\mu) = 0
\mbox{ for each }
\underline h \in \underline H}.
\]
We observe that $\underline H^\perp$ is a left $\End_{\C}(C)$-submodule of $\E(Y,\tau C)$.
Here, for any $f\in\End_{\C}(C)$ and $\mu\in\E(Y,\tau C)$, the action of $f$ on $\mu$ is given by $\tau(\underline{f})_\star\mu$.
Then there exists finitely many $\eta_1,\eta_2,\cdots,\eta_n$ in $\E(Y,\tau C)$ such that $\underline H^\perp = \sum_{i=1}^n \eta_i\End_{\C}(C)$.
Assume that we have an $\s$-triangle
$$\xymatrix@C=0.5cm{\tau C\ar[r]& X_i\ar[r]^{\alpha_i}& Y\ar@{-->}[r]^{\eta_i}&}$$ for each $i=1,2,\cdots,n$.
Then $\alpha_i$ is right $\tau^-\tau C$-determined by Lemma~\ref{lem:ker.det}, and thus $\alpha_i$ is right $C$-determined by Proposition~\ref{prop:det}.
Note that $\bigoplus_{i=1}^n \alpha_i$ is an $\s$-deflation. It is easy to verify that $\bigoplus_{i=1}^n \alpha_i$ is right $C$-determined.

Consider the following commutative diagram
\[\xymatrix{
\bigoplus_{i=1}^n \tau C\ar[r] &X\ar[r]^\alpha\ar[d]  &Y\ar@{-->}[r]^{{\vartriangle}^\star(\bigoplus \eta_i)}&\\
\bigoplus_{i=1}^n \tau C\ar[r] &\bigoplus_{i=1}^n X_i &\bigoplus_{i=1}^n Y
\ar[u];[]^-\vartriangle
\ar[l];[]^{\bigoplus_{i=1}^n \alpha_i}
\ar@{=}[llu];[ll]\ar@{-->}[r]^{\ \bigoplus \eta_i}&,
}\]
where $\vartriangle=(\mbox{Id}_Y,\mbox{Id}_Y,...,\mbox{Id}_Y)^{\mbox{tr}}$.
We have that $\alpha$ is an $\s$-deflation and $\bigoplus_{i=1}^n \tau C\in\add(\tau C)$. By Lemma~\ref{lem:det.PB},
$\alpha$ is right $C$-determined. By a direct verification, we have
\[
\Im \C(C,\alpha)
= \bigcap_{i=1}^n \Im\C(C,\alpha_i).
\]

For each $i=1,2,\cdots,n$, set
\[
^\perp(\eta_i \End_{\C}(C)):
=\set{\underline h \in \Ccu(C,Y) \middle|
\phi_Y(\underline{h})(\mu) = 0
\mbox{ for each }
\mu \in \eta_i \End_{\C}(C)}.
\]
We observe that $^\perp(\eta_i\End_{\C}(C))$ is a right $\End_{\C}(C)$-submodule of $\Ccu(C,Y)$.
Since $\alpha_i$ is an $\s$-deflation, $\CP(C,Y)\subseteq\Im\C(C,\alpha_i)$.

{\bf Claim.} $^\perp(\eta_i\End_{\C}(C))= \Im \C(C,\alpha_i)/\CP(C,Y)$.

Let $h\colon C\to Y$ be a morphism in $\Im\C(C,\alpha_i)$. We have $h^\star\eta_i$ splits. Then
\[
\phi_Y(\underline h)(\tau(\underline f)_\star\eta_i)
= \gamma(h^\star\tau(\underline f)_\star\eta_i)
= \gamma(\tau(\underline f)_\star h^\star\eta_i)
= 0,
\]
for each $f\colon C\to C$. It follows that
$\Im\C(C,\alpha_i)/\CP(C,Y)\subseteq {^\perp(\eta_i\End_{\C}(C))}$.

On the other hand, let $h\in\C(C,Y)$  such that $\underline h\in{^\perp(\eta_i\End_{\C}(C))}$.
Then we have $\phi_Y(\underline h)(\tau(\underline f)_\star\eta_i)=0$ for each $f\colon C\to C$.
Consider the following commutative diagram
\[\xymatrix{
\Ccu(C,Y)\ar[r]^-{\phi_Y} &D\E(Y,\tau C)\\
\Ccu(C,Y)\ar[r]^-{\phi_Y} &D\E(Y,\tau C).
\ar[lu];[l]_{\Ccu(\underline f,Y)}
\ar[u];[]^{D\E(Y,\tau(\underline f))}
}\]
By using diagram chasing, we have
\[\begin{split}
\phi_Y(\underline{h} \circ \underline{f})
&= \phi_Y ( \Ccu(\underline{f}, Y) (\underline{h}) )\\
&= (D\E(Y, \tau(\underline{f})) \circ \phi_Y) (\underline{h})\\
&= \phi_Y(\underline{h}) \circ \E(Y, \tau(\underline{f})).\\
\end{split}\]
Then for each $\eta_i$, we have
\[
\phi_Y(\underline h\circ\underline f)(\eta_i)
= \phi_Y(\underline h)(\tau(\underline f)_\star\eta_i)
= 0.
\]
It follows that
\[
\phi_C(\underline f)(h^\star\eta_i)
= \gamma(f^\star h^\star\eta_i)
= \gamma((h\circ f)^\star\eta_i)
= \phi_Y(\underline{h\circ f})(\eta_i)
= 0.
\]
Observe that $\phi_C(\underline f)$ runs over all maps in $D\E(C,\tau C)$, when $\underline f$ runs over all morphisms
in $\underline\End_\C(C)$. It follows that $h^\star\eta_i$ splits and the morphism $h$ factors through $\alpha_i$.
Then we have $h\in \Im\C(C,\alpha_i)$ and
${^\perp(\eta_i\End_{\C}(C))}\subseteq\Im\C(C,\alpha_i)/\CP(C,Y)$. The claim is proved.

Because
\[
\underline H
= {^\perp(\underline H^\perp)}
= {^\perp(\sum_{i=1}^n\eta_i\End_{\C}(C))}
= \bigcap_{i=1}^n {^\perp(\eta_i\End_{\C}(C))},
\]
where the first equality follows from the isomorphism $\phi_Y$, we have
\[
\underline H
= \bigcap_{i=1}^n \Im \C(C,\alpha_i)/\CP(C,Y)
= \Im \C(C,\alpha)/\CP(C,Y).
  \]
Then the assertion follows since $\CP(C,Y)\subseteq H$.
\end{proof}

\subsection{A characterization for $\s$-deflations determined by objects}

In this section, we give a characterization for an $\s$-deflation being right $C$-determined for some object $C$.

Recall from \cite{Ch15} that two morphisms $f\colon X\to Y$ and $f'\colon X'\to Y$ are called \emph{right equivalent}
if $f$ factors through $f'$ and $f'$ factors through $f$. Assume that $f$ and $f'$ are right equivalent.
Given an object $C$, we have that $f$ is right $C$-determined if and only if so is $f'$.

In what follows,
we always assume that the following {\it weak idempotent completeness} (WIC for short)
given originally in \cite[Condition 5.8]{NP} holds true on $\C$.

{\bf WIC Condition:} \begin{itemize}
\item[$(1)$] Let $f \in \C(A,B)$, $g \in\C(B,C)$ be any composable pair of morphisms. If
$g \circ f$ is an $\s$-inflation, then so is $f$.
\item[$(2)$] Let $f \in \C(A,B)$, $g \in\C(B,C)$ be any composable pair of morphisms. If
$g \circ f$ is an $\s$-deflation, then so is $g$.
\end{itemize}

Under this condition, we have that $f$ is an $\s$-deflation if and only if so is $f'$ provided that $f$ and $f'$ are right equivalent.

Since $\C$ is Krull-Schmidt, each morphism $f\colon X\to Y$ has a right minimal version; see
\cite[Theorem~1]{Bi}.
Given a morphism $f\colon X\to Y$ , we call a right minimal morphism $f'\colon X'\to Y$ the \emph{right minimal version} of $f$,
if $f$ and $f'$ are right equivalent. Assume that $f\colon X\to Y$ is an $\s$-deflation and there exists an $\s$-triangle
$$\xymatrix@C=0.5cm{K\ar[r]&X'\ar[r]^{f'}&Y\ar@{-->}[r]&},$$ then, following \cite[Section~2]{Ri12},
we call $K$ a \emph{intrinsic weak kernel} of $f$.

Dually, two morphisms $f\colon X\to Y$ and $f'\colon X\to Y'$ are called \emph{left equivalent}
if $f$ factors through $f'$ and $f'$ factors through $f$. We have that there exists some left minimal morphism $g\colon X\to X'$
such that $f$ and $g$ are left equivalent. We call $g$ the \emph{left minimal version} of $f$. Assume that
$f\colon X\to Y$ is an $\s$-inflation and there exists an $\s$-triangle
$$\xymatrix@C=0.5cm{X\ar[r]^g&X'\ar[r]&Z\ar@{-->}[r]&},$$ then, we call $Z$ an \emph{intrinsic weak cokernel} of $f$.

\begin{lemma}\label{lem:exist.ass}
Let $a\colon X\to Y$ be an $\s$-deflation and $C$ an indecomposable object. If there exists a morphism $f\colon C\to Y$
which almost factors through $a$, then there exists an almost split $\s$-triangle
$$\xymatrix@C=0.5cm{K\ar[r]&E\ar[r]^r&C\ar@{-->}[r]^\rho&}$$ such that $K$ is a direct summand of an intrinsic weak kernel of $a$.
\end{lemma}

\begin{proof}
We may assume that $a$ is right minimal. Since $a\colon X\to Y$  is an $\s$-deflation, we have an $\s$-triangle
$$\xymatrix@C=0.5cm{L\ar[r]&X\ar[r]^a&Y\ar@{-->}[r]^\sigma&}.$$
Consider a morphism of $\s$-triangles
\[
\xymatrix{
L\ar[r]\ar@{=}[d]&Z\ar[r]^{b}\ar[d]&C\ar@{-->}[r]^{\delta}\ar[d]^f&\\
L\ar[r]&X\ar[r]^{a}&Y\ar@{-->}[r]^{\sigma}&.
}
\]
Then $b$ is an $\s$-deflation. By Proposition~\ref{prop:ras}, $b$ is right almost split.
Let $r\colon E\to C$ be the right minimal version of $b$. Then $r$ is a right almost split $\s$-deflation. Let
$$\xymatrix@C=0.5cm{K\ar[r]&E\ar[r]^r&C\ar@{-->}[r]^{\rho}&}$$ be an $\s$-triangle.
Since $b$ and $r$ are right equivalent, we have the following morphisms of $\s$-triangles
\[
\xymatrix{
K\ar[r]\ar[d]&E\ar[r]^{r}\ar[d]^s&C\ar@{-->}[r]^{\rho}\ar@{=}[d]&\\
L\ar[r]\ar[d]&Z\ar[r]^{b}\ar[d]^t&C\ar@{-->}[r]^{\delta}\ar@{=}[d]&\\
K\ar[r]&E\ar[r]^{r}&C\ar@{-->}[r]^{\rho}&.
}
\]
Since $r$ is right minimal, $t\circ s$ is an isomorphism, and hence $K$ is a direct summand of $L$ by \cite[Corollary 3.6]{NP}.

We claim that
$$\xymatrix@C=0.5cm{K\ar[r]^p&E\ar[r]^r&C\ar@{-->}[r]^{\rho}&}$$ is an almost split $\s$-triangle.
Indeed, assume $f_\star\rho\neq 0$ for any non-section $f \in\C(K,K')$. Let
$$\xymatrix@C=0.5cm{K'\ar[r]&E'\ar[r]^{r'}&C\ar@{-->}[r]^{f_\star\rho}&}$$
be an $\s$-triangle. Consider the following commutative diagram
\[
\xymatrix{
K\ar[r]^p\ar[d]^f&E\ar[r]^{r}\ar[d]&C\ar@{-->}[r]^{\rho}\ar@{=}[d]&\\
K'\ar[r]&E'\ar[r]^{r'}&C\ar@{-->}[r]^{f_\star\rho}&.
}
\]
The $\s$-deflation $r'$ is not a retraction, and hence factors through $r$. We then obtain the following commutative diagram
\[
\xymatrix{
K\ar[r]^p\ar[d]^f&E\ar[r]^{r}\ar[d]^g&C\ar@{-->}[r]^{\rho}\ar@{=}[d]&\\
K'\ar[r]\ar[d]^{f'}&E'\ar[r]^{r'}\ar[d]^{g'}&C\ar@{-->}[r]^{f_\star\rho}\ar@{=}[d]&\\
K\ar[r]^p&E\ar[r]^{r}&C\ar@{-->}[r]^{\rho}&.
}
\]
Since $r$ is right minimal, the morphism $g' \circ g$ is an isomorphism. We obtain that $f' \circ f$ is an isomorphism,
which is a contradiction since $f$ is not a section. This shows that $f_\star\rho= 0$  for any non-section $f \in\C(K,K')$.
By the same argument, we have $g^\star \rho=0$ for any non-retraction $g\in\C(C',C)$.
The claim is proved.
\end{proof}

By using an argument similar to that in the proofs of \cite[Proposition~XI.2.4 and Lemma~XI.2.1]{Au95}, we get that
if $C$ is a minimal right determiner of an $\s$-deflation $\alpha\colon X\to Y$, then an indecomposable object $C'$
is a direct summand of $C$ if and only if there exists a morphism $f\colon C'\to Y$ which almost factors through $\alpha$.

\begin{corollary}\label{cor:min.det}
A minimal right determiner of an $\s$-deflation has no non-zero $\s$-projective direct summands and lies in $\C_r$.
\end{corollary}

\begin{proof}
Let $\alpha\in\C(X,Y)$ be an $\s$-deflation and $C$ a minimal right determiner of $\alpha$. It suffices to show that each indecomposable
direct summand $C'$ of $C$ is not $\s$-projective and lies in $\C_r$. By assumption, there exists a morphism $f\colon C'\to Y$
which almost factors through $\alpha$. We have that $C'$ is not $\s$-projective, since $f$ does not factor through $\alpha$.
By Lemma~\ref{lem:exist.ass}, there exists an almost split $\s$-triangle ending at $C'$. Then the assertion
follows from Proposition~\ref{prop:C_r and C_l description}.
\end{proof}

Now we  give a characterization for an $\s$-deflation being right $C$-determined for some object $C$.

\begin{theorem}\label{thm:det}
For any $\s$-deflation $\alpha\in\C(X,Y)$, the following statements are equivalent.
\begin{enumerate}
\item[$(1)$] $\alpha$ is right $C$-determined for some object $C$.
\item[$(2)$] The intrinsic weak kernel of $\alpha$ lies in $\C_l$.
\end{enumerate}
\end{theorem}

\begin{proof}
$(1)\Rightarrow (2)$  We may assume that $C$ is a minimal right determiner of $\alpha$. By Corollary~\ref{cor:min.det},
we have $C\in\C_r$. We observe that $\Im\C(C,\alpha)$ is a right $\End(C)$-submodule of $\C(C,Y)$. Since $\alpha$ is an $\s$-deflation,
we have $\CP(C,Y)\subseteq\Im\C(C,\alpha)$. By Theorem~\ref{thm:exist}, there exists an $\s$-triangle
$$\xymatrix@C=0.5cm{K\ar[r]&Z\ar[r]^{\beta}&Y\ar@{-->}[r]^{\rho}&}$$ such that $\beta$ is right $C$-determined,
$K\in\add(\tau C)$ and $\Im\C(C,\alpha)=\Im\C(C,\beta)$. For any $f\in\C(C,X)$ and $g\in\C(C,X')$, we have that
$\alpha\circ f$ and $\beta\circ g$ factor through $\beta$ and $\alpha$ respectively. Since $\alpha$ and $\beta$ are right $C$-determined,
$\alpha$ and $\beta$ factor through each other. It follows that $\alpha$ and $\beta$ are right equivalent.

Let $\alpha'\colon X'\to Y$ be the right minimal version of $\alpha$. Then $\alpha'$ and $\beta$ are right equivalent.  Let
$$\xymatrix@C=0.5cm{K'\ar[r]&X'\ar[r]^{\alpha'}&Y\ar@{-->}[r]^{\delta}&}$$ be an $\s$-triangle.
Then $K'$ is an intrinsic weak kernel of $\alpha'$. Since $\alpha'$ and $\beta$ are right equivalent, there exist $s: X'\to Z$ and $t: Z\to X'$
such that $\alpha'=\beta \circ s$ and $\beta=\alpha'\circ t$, and hence we have the following morphisms of $\s$-triangles
\[
\xymatrix{
K'\ar[r]\ar[d]^{s'}&X'\ar[r]^{\alpha'}\ar[d]^s&Y\ar@{-->}[r]^{\delta}\ar@{=}[d]&\\
K\ar[r]\ar[d]^{t'}&Z\ar[r]^{\beta}\ar[d]^t&Y\ar@{-->}[r]^{\rho}\ar@{=}[d]&\\
K'\ar[r]&X'\ar[r]^{\alpha'}&Y\ar@{-->}[r]^{\delta}&.
}
\]
Since $\alpha'$ is right minimal, $t\circ s$ is isomorphic, and so is $t'\circ s'$. Thus $K'$ is a direct summand of $K$.
The assertion follows since $K \in \add (\tau C)$ and $\tau C \in \C_l$.

$(2)\Rightarrow (1)$  Let $\alpha'\colon X'\to Y$ be the right minimal version of $\alpha$. Then $\alpha'$ is an $\s$-deflation. Let
$$\xymatrix@C=0.5cm{K'\ar[r]&X'\ar[r]^{\alpha'}&Y\ar@{-->}[r]^{\delta}&}$$ be an $\s$-triangle. Then $K'\in\C_l$ by (2).
By Lemma~\ref{lem:ker.det}, we have that $\alpha'$ is right $\tau^-(K')$-determined. It follows that $\alpha$ is right $\tau^-(K')$-determined
since $\alpha$ and $\alpha'$ are right equivalent.
\end{proof}

\subsection{Characterizations for objects in $\C_r$}

In this subsection, we will give some characterizations for an object lying in $\C_r$ via morphisms determined by objects.

\begin{proposition}\label{prop:C_l}
For an object $K$ without non-zero $\s$-injective direct summands, the following statements are equivalent.
\begin{enumerate}
\item[$(1)$] $K\in\C_l$.
\item[$(2)$] $K$ is an intrinsic weak kernel of some  $\s$-deflation $\alpha\colon X\to Y$, which is right $C$-determined for some object $C$.
\end{enumerate}
\end{proposition}

\begin{proof}
$(1)\Rightarrow (2)$ Let $K\in\C_l$. Decompose $K$ as the direct sum of indecomposable objects $K_1,K_2,\cdots,K_n$.
We have that each $K_i$ is non-$\s$-injective. Then, by Proposition \ref{prop:C_r and C_l description}, for each $i=1,2,\cdots,n$,
there exists an almost split $\s$-triangle
$$\xymatrix@C=0.5cm{K_i\ar[r]&X_i\ar[r]^{\alpha_i}&Y_i\ar@{-->}[r]^{\rho_i}&}.$$
We have that $\bigoplus_{i=1}^n \alpha_i$ is an $\s$-deflation. It follows from Lemma~\ref{lem:ker.det} that
$\bigoplus_{i=1}^n \alpha_i$ is right $\tau^-K$-determined. Moreover, since $\alpha_i$ is right minimal, so is $\bigoplus_{i=1}^n \alpha_i$.
It follows that $K$ is an intrinsic weak kernel of $\bigoplus_{i=1}^n \alpha_i$.

$(2)\Rightarrow (1)$ It follows from Theorem~\ref{thm:det}.
\end{proof}

The following lemma is the converse of Theorem~\ref{thm:exist}.

\begin{lemma}\label{lem:C_r}
Let $C\in\C$. If for each $Y\in\C$ and each right $\End(C)$-submodule $H$ of $\C(C,Y)$ satisfying $\CP(C,Y)\subseteq H$,
there exists a right $C$-determined $\s$-deflation $\alpha\colon X\to Y$ such that $\Im\C(C,\alpha)=H$,
then $C\in\C_r$.
\end{lemma}

\begin{proof}
It suffices to show that each non-$\s$-projective indecomposable direct summand $C'$ of $C$ lies in $\C_r$.

We claim that each $f\in\CP(C,C')$ is not a retraction. Otherwise, if some $f\in\CP(C,C')$
is a retraction, then there exists $g\in\C(C',C)$ such that $fg={\rm Id}_{C'}$.
Since $f$ is $\s$-projective, we have that ${\rm Id}_{C'}=fg$ is $\s$-projective, and hence
$C'$ is an $\s$-projective object, which is a contradiction. The claim is proved.
Notice that $\radc(C,C')$ is formed by non-retractions,
so $\CP(C,C')\subseteq\radc(C,C')$. Because $\radc(C,C')$ is a right $\End(C)$-submodule of $\C(C,C')$, by  assumption
there exists a right $C$-determined $\s$-deflation $\alpha\colon X\to C'$ such that $\radc(C,C')=\Im\C(C,\alpha)$.

Since $\Im\C(C,\alpha) = \radc(C,C')$ is a proper submodule of $\C(C,C')$, the $\s$-deflation $\alpha$ is not a retraction.
Thus $\Id_{C'}$ does not factor through $\alpha$. Let $f\in\radc(T,C')$. For each $g\colon C\to T$,
the morphism $f\circ g$ lies in $\radc(C,C')=\Im\C(C,\alpha)$. It follows that $f\circ g$ factors through $\alpha$.
Since $\alpha$ is right $C$-determined, we have that $f$ factors through $\alpha$ and $\Id_{C'}$ almost factors through $\alpha$.
By Lemma~\ref{lem:exist.ass}, there exists an almost split $\s$-triangle ending at $C'$.
Now the assertion follows from Proposition~\ref{prop:C_r and C_l description}.
\end{proof}

Collecting the results obtained so far, we list some characterizations for an object lying in $\C_r$.

\begin{theorem}\label{thm:C}
For any $C\in\C$, the following statements are equivalent.
\begin{enumerate}
\item[$(1)$]\label{item:thm:C:1}
$C\in\C_r$.
\item[$(2)$]\label{item:thm:C:2}
For each $Y\in\C$ and each right $\End(C)$-submodule $H$ of $\C(C,Y)$ satisfying $\CP(C,Y)\subseteq H$, there exists a
right $C$-determined $\s$-deflation $\alpha\colon X\to Y$ such that $H=\Im\C(C,\alpha)$.
\end{enumerate}
If moreover $C$ is non-$\s$-projective indecomposable, then all statements (1)--(6) are equivalent.
\begin{enumerate}
\setcounter{enumi}{3}
\item[$(3)$]\label{item:thm:C:4}
$C$ is an intrinsic weak cokernel of some $\s$-inflation $\alpha\colon X\to Y$ which is  left $K$-determined for some object $K$.
\item[$(4)$]\label{item:thm:C:5}
There exists an almost split $\s$-triangle ending at $C$.
\item[$(5)$]\label{item:thm:C:6}
There exists a non-retraction $\s$-deflation which is right $C$-determined.
\item[$(6)$]\label{item:thm:C:7}
There exists an $\s$-deflation $\alpha\colon X\to Y$ and a morphism $f\colon C\to Y$ such that $f$ almost factors through $\alpha$.
\end{enumerate}
\end{theorem}

\begin{proof}
By Theorem~\ref{thm:exist} and Lemma~\ref{lem:C_r}, we have (1) $\Leftrightarrow$ (2).

Assume that $C$ is non-$\s$-projective indecomposable. Then the dual of Proposition~\ref{prop:C_l} implies
$(1)\Leftrightarrow (3)$, and Proposition~\ref{prop:C_r and C_l description} implies $(1)\Leftrightarrow (4)$.
By Lemma~\ref{lem:exist.ass}, we have $(6)\Rightarrow (4)$.

It is easy to see that the right almost split $\s$-deflation ending at $C$ is a non-retraction and right $C$-determined. Then we have
$(4)\Rightarrow (5)$. Let $\alpha$ be a right $C$-determined $\s$-deflation which is not a retraction. We have that $C$ is a minimal
right determiner of $\alpha$. Thus $(5)\Rightarrow (6)$ holds true.
\end{proof}

\section{Examples}

\subsection{}

In \cite{INP}, Iyama, Nakaoka and Palu introduced the following notions in order to study the existence of almost split extensions.
Let $(\C,\E,\s)$ be a $k$-linear extriangulated category.

\begin{enumerate}[(a)]
 \item[(1)] A \emph{right Auslander-Reiten-Serre duality} is a pair $(\tau,\eta)$ of an additive functor
 $\tau\colon\underline{\C}\to\overline{\C}$ and a binatural isomorphism
\[\eta_{A,B}\colon\underline{\C}(A,B)\simeq D\E(B,\tau A)\ \mbox{ for any }\ A,B\in\C.\]
 \item[(2)] If moreover $\tau$ is an equivalence, then $(\tau,\eta)$ is called an \emph{Auslander-Reiten-Serre duality}.
\end{enumerate}

Following \cite[Proposition 3.5 and Theorem 3.4]{INP}, we have

\begin{proposition} Let $\C$ be a $k$-linear Ext-finite Krull-Schmidt extriangulated category.
\begin{enumerate}
\item[$(1)$] $\C=\C_r$ if and only if $\C$ has a right Auslander-Reiten-Serre duality $(\tau,\eta)$,
in which $\tau$ is fully faithful.
\item[$(2)$] $\C=\C_r=\C_l$ if and only if $\C$ has an Auslander-Reiten-Serre duality. In this case,
the Auslander-Reiten-Serre duality is exactly the pair $(\tau,\psi)$ defined in Section 3.2.
\end{enumerate}
\end{proposition}

Let $A$ be a finite-dimensional algebra over a field $k$ and $A\text{-mod}$ the category
of finitely generated left $A$-modules. We use $\mathcal{P}(A)$ to denote the subcategory of $A\text{-mod}$
consisting of projective modules, and use ${\mathcal{GP}}(A)$ to denote the subcategory of $A\text{-mod}$
consisting of Gorenstein projective modules.

\begin{example}
\begin{itemize}
\item[]
\item[$(1)$] It is well known that $A\text{-mod}$ has an Auslander-Reiten-Serre duality.
Moreover, if $A$ is self-injective, then the stable category $A$-$\underline{{\rm mod}}$ has
an Auslander-Reiten-Serre duality (\cite{Hap}).
\item[$(2)$] If $A$ is Gorenstein (that is, the left and right self-injective dimensions of $A$ are finite),
then the stable category $\underline{\mathcal{GP}}(A)$
has an Auslander-Reiten-Serre duality. In fact, since ${\mathcal{GP}}(A)$ is an extension-closed
functorially finite subcategory of $A\text{-mod}$, ${\mathcal{GP}}(A)$ has almost split sequences,
and they induce almost split triangle in $\underline{\mathcal{GP}}(A)$. Moreover,
if $A$ is Gorenstein and ${\mathcal{GP}}(A)$ is of finite type, then the Gorenstein derived category $D^b_{gp}(A)$
has an Auslander-Reiten-Serre duality (\cite{Gao}).
\item[$(3)$] If the global dimension of $A$ is finite, then the bounded derived category
$D^b(A)$  has an Auslander-Reiten-Serre duality (\cite{Hap}).
\item[$(4)$] If $\mathcal{X}$ is an extension-closed functorially finite subcategory of the bounded homotopy category
$K^b(\mathcal{P}(A))$, then $\mathcal{X}$ has an Auslander-Reiten-Serre duality \cite[Proposition 6.1]{INP}.
\end{itemize}
\end{example}

\subsection{}
Let $Q=(Q_0,Q_1)$ be a quiver, where $Q_0$ consists of all vertices, and $Q_1$ consists of all arrows.
For $x\in Q_0$, we use the symbol $x^+$ (resp. $x^-$) to denote the set of arrows starting (resp. ending) with $x$.
For $x,y\in Q_0$, let $Q(x,y)$ stand for the set of paths in $Q$ from $x$ to $y$. Then
\begin{itemize}
\item $Q$ is called \emph{locally finite} if $x^+$ and $x^-$ are finite for each $x\in Q_0$.
\item $Q$ is called \emph{interval-finite} if $Q(x,y)$ is finite for each pair $x,y\in Q_0$.
\item $Q$ is called \emph{strongly locally finite} if it is locally finite and interval-finite.
\end{itemize}

Let $k$ be a field. A $k$-representation $M$ of a quiver $Q$ consists of a family of $k$-spaces $M_x$
with $x\in Q_0$, and a family of $k$-maps $M_\alpha: M_x\rightarrow M_y$ with $\alpha:x\rightarrow y$ in $Q_1$.
A $k$-representation $M$ is called {\it locally finite-dimensional} if $M_x$ is of finite
$k$-dimensional for each $x\in Q_0$, and called {\it finite-dimensional} if $\sum_{x\in Q_0}{\rm dim}M_x<\infty$.

In the following, we will use frequently three kinds of $k$-representations $S_a$, $P_a$ and $I_a$ which are defined by the following way:
\begin{itemize}
\item $S_a$ is the $k$-representation such that $S_a(a)=k$ and $S_a(x)=0$ for any $x\in Q_0\setminus \{a\}$.
\item $P_a$ is the $k$-representation such that $P_a(x)$, for any $x\in Q_0$, is the $k$-space spanned by $Q(a,x)$ and
$P_a(\alpha):P_a(x)\to P_a(y)$, for any $\alpha: x\to y\in Q_1$, is the $k$-map sending every path $p$ to $\alpha p$.
\item $I_a$ is the $k$-representation such that $I_a(x)$, for any $x\in Q_0$, is the $k$-space spanned by $Q(x,a)$ and
$I_a(\alpha):I_a(x)\to I_a(y)$, for any $\alpha: x\to y\in Q_1$, is the $k$-map sending every $p\alpha$ to $p$ and vanishing on the
paths which do not factor through $\alpha$.
\end{itemize}

We use the following notations:
\begin{itemize}
\item $\mbox{Rep}(Q):=$the category of all $k$-representations of $Q$.
\item $\mbox{rep}(Q):=$the full subcategory of $\mbox{Rep}(Q)$ consisting of all locally finite-dimensional $k$-representations of $Q$.
\item $\mbox{Inj}(Q):=$the full additive subcategory of  $\mbox{Rep}(Q)$ generated by the  objects isomorphic to $I_a\otimes V_a$
with $a\in Q_0$ and $V_a$ some $k$-space.
\item $\mbox{Proj}(Q):=$the full additive subcategory of  $\mbox{Rep}(Q)$ generated by the  objects isomorphic to $P_a\otimes U_a$
with $a\in Q_0$ and $U_a$ some $k$-space.
\item $\mbox{inj}(Q):=\mbox{Inj}(Q)\cap \mbox{rep}(Q)$ and $\mbox{proj}(Q):=\mbox{Proj}(Q)\cap \mbox{rep}(Q)$.
\item $\mbox{rep}^+(Q):=$the full subcategory of $\mbox{rep}(Q)$ consisting of all finitely presented $k$-representations of $Q$.
\item $\mbox{rep}^-(Q):=$the full subcategory of $\mbox{rep}(Q)$ consisting of all finitely co-presented $k$-representations of $Q$.
\end{itemize}
Here a $k$-representation $M$ is called \emph{finitely presented} if there exists an exact sequence
$$0\to P_1\to P_0\to M\to 0$$ with $P_0,P_1\in\mbox{proj}(Q)$, and \emph{finitely co-presented}
if there exists an exact sequence
$$0\to M\to I_0\to I_1\to 0$$ with $I_0,I_1\in\mbox{inj}(Q)$.

By \cite[(1.15)]{BLP}, ${\rm rep}^+(Q)$ and ${\rm rep}^-(Q)$ are Hom-finite $k$-linear abelian categories,
and hence they are extriangulated categories in which the corresponding $\E$ and $\s$ were given in \cite[Example 2.13]{NP}.

Following \cite[Theorem 5.2]{LNP} and  Proposition \ref{prop:C_r and C_l description}, we have
\begin{theorem}
Let $Q$  be a strongly locally finite quiver (for example, type of $\mathbb{A}_\infty$ or $\mathbb{A}^\infty_\infty$)
and $M$ an indecomposable $k$-representation in  ${\rm rep}(Q)$.
\begin{itemize}
\item [$(1)$] If $M\in{\rm rep}^+(Q)$ is not projective, then $M\in{\rm rep}^+(Q)_r$ if and only if $\tau M$ is finite-dimensional.
\item [$(2)$] If $N\in{\rm rep}^-(Q)$ is not injective, then $N\in{\rm rep}^-(Q)_l$ if and only if $\tau^- M$ is finite-dimensional.
\end{itemize}
Here $\tau$ and $\tau^-$ are defined to be the Auslander-Reiten translations $D{\rm Tr}$ and ${\rm Tr}D$ (see \cite[Definition 2.4]{BLP}).
\end{theorem}

Now let $Q$ be a quiver of type $\mathbb{A}_\infty$ as follows:
$$
1\rightarrow 2\leftarrow 3\rightarrow 4\rightarrow 5\rightarrow \cdots
$$
and $\C={\rm rep}^+(Q)$.
\begin{itemize}
\item [(1)]Let $C=S_1$ be non-projective indecomposable. Since $\tau S_1=S_2$ is finite-dimensional,
we have $S_1\in{\rm rep}^+(Q)_r$. Moreover, there exists an almost split sequence
$$\xymatrix@C=0.5cm{0\ar[r]&S_2\ar[r]^{\alpha}&P_1\ar[r]&S_1\ar[r]&0},$$
and $\alpha$ is left $S_2$-determined by the dual of Lemma \ref{lem:ker.det}.
\item [(2)]Clearly, $P_4$ is not injective in ${\rm rep}^+(Q)$. By \cite[Proposition 3.6]{BLP},
there exists no almost split sequence starting with $P_4$ in ${\rm rep}^+(Q)$. Thus, by
Proposition \ref{prop:C_r and C_l description}, $P_4\notin {{\rm rep}^+(Q)}_l$.
Consider the canonical epimorphism $p:P_4\rightarrow S_4$. It is easy to see that the intrinsic weak kernel of
$p$ is $P_5$ and $P_5\notin {{\rm rep}^+(Q)}_l$. Thus, by Theorem \ref{thm:det},
$p$ could not be right determined by any object in ${\rm rep}^+(Q)$.
\end{itemize}

Now we let $\C=D^b({\rm rep}^+(Q))$ be the derived category of the bounded complexes in ${\rm rep}^+(Q)$.
Following \cite[Theorem 7.11]{BLP}, we have

\begin{theorem}
\begin{itemize}
\item[]
\item[$(1)$] $D^b({\rm rep}^+(Q))_l=D^b({\rm rep}^+(Q))$ if and only if $Q$ has no left infinite path.
\item[$(2)$] $D^b({\rm rep}^+(Q))_r=D^b({\rm rep}^+(Q))$ if and only if $Q$ has no right infinite path.
\item[$(3)$] $D^b({\rm rep}^+(Q))_l=D^b({\rm rep}^+(Q))=D^b({\rm rep}^+(Q))_r$ if and only if $Q$ has no infinite path.
\end{itemize}
\end{theorem}

\vspace{0.5cm}

{\bf Acknowledgements.} This work was partially supported by National Natural Science Foundation of China
(Nos. 11971225, 11571164, 11901341), a Project Funded
by the Priority Academic Program Development of Jiangsu Higher Education Institutions,
  the project ZR2019QA015 supported by Shandong Provincial Natural Science Foundation, and the Young Talents Invitation Program of Shandong Province. The first author thanks Pengjie Jiao for his help, and all authors thank the referee
for very useful suggestions.

\end{document}